\documentclass[leqno,11pt]{amsart}
\usepackage{amsmath, amssymb, latexsym}
\usepackage{mathrsfs}
\usepackage{enumerate}
\usepackage[all, knot]{xy}
\usepackage{color}

\newtheorem{thm}{Theorem}[section]
\newtheorem{prop}[thm]{Proposition}
\newtheorem{cor}[thm]{Corollary}
\newtheorem{lemma}[thm]{Lemma}

\theoremstyle{definition}
\newtheorem{defn}[thm]{Definition}

\theoremstyle{remark}

\newenvironment{red}
{\relax\color{red}}
{\hspace*{.5ex}\relax}

\newcommand{\ber}{\begin{red}}
\newcommand{\er}{\end{red}}

\newenvironment{verd}
{\relax\color{magenta}}
{\hspace*{.5ex}\relax}

\newcommand{\bg}{\begin{verd}}
\newcommand{\eg}{\end{verd}}

\numberwithin{equation}{section}

\newcommand{\Z}{\mathbb{Z}}
\newcommand{\Q}{\mathbb{Q}}
\newcommand{\C}{\mathbb{C}}
\newcommand{\A}{\mathbb{A}}

\newcommand{\g}{\mathfrak{g}}

\newcommand{\Hom}{\mathrm{Hom}}
\newcommand{\End}{\mathrm{End}}
\newcommand{\Ext}{\mathrm{Ext}}

\newcommand{\wt}{{\rm wt}}

\newcommand{\Top}{{\rm Top}}
\newcommand{\Soc}{{\rm Soc}}
\newcommand{\Rad}{{\rm Rad}}

\renewcommand{\mod}{{\rm \text{-}mod}}

\newcommand{\proj}{\mathrm{proj}}

\newcommand{\Ker}{\mathrm{Ker}}

\newcommand{\rlQ}{\mathsf{Q}}   
\newcommand{\wlP}{\mathsf{P}}   
\newcommand{\weyl}{\mathsf{W}}  
\newcommand{\cmA}{\mathsf{A}}  
\newcommand{\kf}{\tilde{f}}  
\newcommand{\ke}{\tilde{e}}  

\newcommand{\ST}{\mathsf{ST}}   
\newcommand{\res}{\mathrm{res}}   
\newcommand{\F}{\mathcal{F}}   

\newcommand{\bR}{\mathbf{k}}   
\newcommand{\fqH}{R^{\Lambda_0}}   

\newcommand{\codeg}{\mathrm{codeg}}

\renewcommand{\Im}{\operatorname{Im}}
\renewcommand{\Ker}{\operatorname{Ker}}

\begin{document}

\title[Representation type of finite quiver Hecke algebras of type $C^{(1)}_{\ell}$]
{Representation type of finite quiver \\
Hecke algebras of type $C^{(1)}_{\ell}$}

\author[Susumu Ariki]{Susumu Ariki}
\address{Department of Pure and Applied Mathematics, Graduate School of Information
Science and Technology, Osaka University, Toyonaka, Osaka 560-0043, Japan}
\email{ariki@ist.osaka-u.ac.jp}

\author[Euiyong Park]{Euiyong Park$^1$ }
\thanks{$^1$ E.P. is supported by the National Research Foundation of Korea(NRF) Grant funded by the Korean Government(MSIP)(No.\ 2014R1A1A1002178).}
\address{Department of Mathematics, University of Seoul, Seoul 130-743, Korea}
\email{epark@uos.ac.kr}


\begin{abstract}
We give a graded dimension formula described in terms of combinatorics of Young diagrams and a simple criterion to determine the representation type for the finite quiver Hecke algebras of type $C_{\ell}^{(1)}$.
\end{abstract}

\maketitle


\vskip 2em

\section*{Introduction}

This is the fourth of our series on finite quiver Hecke algebras.
The {\it quiver Hecke algebras}, or {\it affine quiver Hecke algebras}, were introduced by Khovanov-Lauda \cite{KL09, KL11} and Rouquier \cite{R08} for providing
categorification of (the negative half of) quantum groups.
Their certain quotient algebras, the {\it cyclotomic quiver Hecke algebras} $R^\Lambda(\beta)$, where
$\Lambda$ is fixed and $\beta$ is varying, together with induction and restriction functors among their module categories,
categorify the irreducible highest weight module $V(\Lambda)$ over the quantum group.
When $\Lambda = \Lambda_0$, we call the algebras $\fqH(\beta)$ the {\it finite quiver Hecke algebras}.
As was explained in our previous papers \cite{AIP13,AP12,AP13} in the series, finite quiver Hecke algebras
can be understood as vast generalization of the Iwahori-Hecke algebras associated with the symmetric group in the direction of Lie type.

In this paper, we study the representation type of finite quiver Hecke algebra $\fqH(\beta)$ of affine type $C_\ell^{(1)}$.
The main results are a graded dimension formula of $\fqH(\beta)$ described in terms of combinatorics of Young diagrams (Theorem \ref{Thm: dimension forumula}) and
a criterion for the representation type of $\fqH(\beta)$ in Lie theoretic terms (Theorem \ref{Thm: main thm}).
Recall that we studied affine types $A^{(1)}_\ell$, $A^{(2)}_{2\ell}$ and $D^{(2)}_{\ell+1}$ in our previous papers, and proved that
the patterns of the representation type followed natural generalization of
Erdmann and Nakano's for the Iwahori-Hecke algebras associated with the symmetric group.
However, the affine type $C_\ell^{(1)}$ shows a new pattern.
In particular, we have an unexpected result that $\fqH(\delta)$ is not of finite representation type.

Now, we explain in some detail the tools and the strategy to prove the results. Firstly,
the $q$-deformed Fock space $\F$ of type $C_\ell^{(1)}$ \cite{KMM93} is a key ingredient for proving the graded dimension formula.
This $C_\ell^{(1)}$-type Fock space $\F$ is constructed by folding the usual $q$-deformed
$A_{2\ell-1}^{(1)}$- type Fock space. Namely, the basis is given by the set of all partitions as in the usual Fock space,
but we change the residue pattern on the nodes of partitions via the folding map
$$ \pi: \{0,1, \ldots, 2 \ell-1 \} \rightarrow \{0,1, \ldots, \ell \} $$
defined by $\pi(0)=0$, $\pi(\ell)=\ell$ and $\pi(2\ell-i)=\pi(i)=i$ for $i = 1 , \ldots, \ell-1$.
Investigating the action of $e_{\nu_1} \cdots e_{\nu_n} f_{\nu_n'} \cdots f_{\nu_1'}$ on the Fock space $\F$, we obtain the dimension formula.
Thus, the formula is described in terms of combinatorics of Young diagrams, which is very similar to the graded dimension formula of affine type $A$ in \cite[Sec.\ 4.11]{BK2-09}.
We remark that the residue pattern $\eqref{Eq: residue seq}$ for type $C_\ell^{(1)}$ also appears as colors of arrows in the Kirillov-Reshetikhin crystal $B^{1,1}$ of type $C_\ell^{(1)}$, which is not a perfect crystal \cite{FOS09}.

To achieve the second result, we follow the framework to determine the representation type given in \cite{AP12}.
Let $\max(\Lambda)$ denote the set of maximal weights of the irreducible highest weight module $V(\Lambda)$.
In the three affine cases studied in our previous papers in the series, the set $\max(\Lambda_0)$ consists of
a single Weyl group orbit. Thus, we may generalize the notion of cores and weights of Young diagrams. In the affine type $C_\ell^{(1)}$,
$\max(\Lambda_0)$ consists of several Weyl group orbits and the representatives are given by
the set $\max(\Lambda_0) \cap \wlP^+$. It is not difficult to calculate the set and the result is
$$ \max(\Lambda_0) \cap \wlP^+ = \{  \Lambda_0 + \varpi_i - \frac{i}{2}\delta \mid i\in I,\ \text{$i$ is even } \}, $$
where $\varpi_0=0$ and if $i\ne 0$ then
$$ \varpi_i = \alpha_1 + 2\alpha_2 + \ldots + (i-1)\alpha_{i-1} + i( \alpha_i + \alpha_{i+1} + \cdots + \alpha_{\ell-1} + \frac{1}{2}\alpha_\ell  ). $$
Thus, by the $sl_2$-categorification theorem, we have to investigate the representation type
of $\fqH(k\delta - \varpi_i)$ for $k \ge \frac{i}{2} $. We first consider the representation type of $\fqH(\delta)$.

Recall that one of the ingredients in our series of papers was explicit construction of $\fqH(\delta)$-modules
or $\fqH(2\delta)$-modules. Recently, an interesting paper by Kleshchev and Muth \cite{KM13} appeared, and they
constructed irreducible $\fqH(\delta)$-modules for several untwisted affine types in the spirit of Kang, Kashiwara and Kim \cite{KKK13}, which includes
the affine type $C_\ell^{(1)}$. Thus, we use their construction and, combining with the dimension formula,
we find the radical series of the indecomposable projective $\fqH(\delta)$-modules, and determine the representation type of $\fqH(\delta)$ (Theorem \ref{Thm: wild for delta}).
The result is that $\fqH(\delta)$ is a symmetric special biserial algebra if $\ell=2$, and it is of wild representation type if $\ell \ge 3$.

Next task is to deal with the representation type of $\fqH(2\delta - \varpi_4)$.
In this case, we do not need explicit description of irreducible modules, and
we may derive the radical series of the indecomposable projective modules from the categorification theorem and crystal properties.
The result tells that $\fqH(2\delta - \varpi_4)$ is of wild representation type (Theorem \ref{Thm: wild for 2delta-pi4}).

Using the same arguments in \cite{AP12} with small modifications, we may handle the remaining cases, and we
obtain the second main result (Theorem \ref{Thm: main thm}).

\vskip 1em

\section{Quantum affine algebras} \label{Sec: Fock space}

Let $I = \{0,1, \ldots, \ell \}$ be an index set, and $\cmA$ the {\it affine Cartan matrix} of type $C_{\ell}^{(1)}$ ($\ell \ge 2$)
$$ \cmA = (a_{ij})_{i,j\in I} = \left(
                                  \begin{array}{ccccccc}
                                    2  & -1 & 0  & \ldots & 0  & 0  & 0 \\
                                    -2 &  2 & -1 & \ldots & 0  & 0  & 0 \\
                                    0  & -1 & 2  & \ldots & 0  & 0  & 0 \\
                                    \vdots   &  \vdots  &  \vdots  & \ddots &  \vdots  &  \vdots  & \vdots \\
                                    0  & 0  & 0  & \ldots & 2  & -1 & 0 \\
                                    0  & 0  & 0  & \ldots & -1 & 2  & -2 \\
                                    0  & 0  & 0  & \ldots & 0  & -1 & 2 \\
                                  \end{array}
                                \right).
  $$
\footnotetext{If $\ell=1$ then it becomes the affine type $A^{(1)}_1$, which was already studied in \cite{AIP13}.}

An {\it affine Cartan datum} $(\cmA, \wlP, \Pi, \Pi^{\vee})$ of type $C_\ell^{(1)}$ consists of
\begin{itemize}
\item[(1)] the affine Cartan matrix $\cmA$ as above,
\item[(2)] a free abelian group $\wlP$ of rank $\ell+2$, called the {\it weight lattice},
\item[(3)] $\Pi = \{ \alpha_i \mid i\in I \} \subset \wlP$, called the set of {\it simple roots},
\item[(4)] $\Pi^{\vee} = \{ h_i \mid i\in I\} \subset \wlP^{\vee} := \Hom( \wlP, \Z )$, called the set of {\it simple coroots},
\end{itemize}
which satisfy the following properties:
\begin{itemize}
\item[(a)] $\langle h_i, \alpha_j \rangle  = a_{ij}$ for all $i,j\in I$,
\item[(b)] $\Pi$ and $\Pi^{\vee}$ are linearly independent sets.
\end{itemize}

The free abelian group $\rlQ = \bigoplus_{i \in I} \Z \alpha_i$ is called the {\it root lattice},  and $\rlQ^+ = \sum_{i\in I} \Z_{\ge 0} \alpha_i$ is
the {\it positive cone} of the root lattice. For $\beta=\sum_{i \in I} k_i \alpha_i \in \rlQ^+$, set $|\beta|=\sum_{i \in I} k_i$ to be the {\it height} of $\beta$.
We denote by $\weyl$ the {\it Weyl group} associated with $\cmA$, which is generated by $\{r_i\}_{i\in I}$ acting on $\wlP$ by
$r_i\Lambda=\Lambda-\langle h_i, \Lambda\rangle\alpha_i$, for $\Lambda\in \wlP$.
Let
$$ \wlP^+ = \{ \Lambda \in \wlP \mid \Lambda( h_i ) \ge 0 \text{ for } i\in I \}.  $$
For $i\in I$, let $\Lambda_i$ be the $i$th {\it fundamental weight} in $\wlP^+$. In particular, we have $\Lambda_i(h_j) = \delta_{i,j}$.
The \emph{null root} in the affine type $C^{(1)}_{\ell}$ is given by
$$ \delta = \alpha_0 + 2\alpha_1 + \cdots + 2\alpha_{\ell-1} + \alpha_\ell. $$
Note that $ \langle h_i, \delta\rangle = 0 $ and $w \delta = \delta$, for $i\in I$ and $w\in \weyl$.
Let $(\mathsf{d}_0, \mathsf{d}_1, \ldots, \mathsf{d}_\ell) = (2,1, \ldots, 1,2)$. Then the standard symmetric bilinear
pairing $(\ | \ )$ on $\wlP$ satisfies
\begin{align} \label{Eq: d}
( \alpha_i | \Lambda ) = \mathsf{d}_i \langle h_i , \Lambda \rangle \ \text{ for all } \Lambda \in \wlP.
\end{align}
We set $\varpi_0 := 0$, and we define, for $i \in I \setminus \{ 0 \} $,
\begin{align} \label{Eq: pi}
\varpi_i := \alpha_1 + 2\alpha_2 + \ldots + (i-1)\alpha_{i-1} + i( \alpha_i + \alpha_{i+1} + \cdots + \alpha_{\ell-1} + \frac{1}{2}\alpha_\ell  ).
\end{align}
Note that if $i\ne0$ then
$$\varpi_i(h_j) = \left\{
                              \begin{array}{ll}
                                -1 & \hbox{ if } j=0, \\
                                1 & \hbox{ if } j=i, \\
                                0 & \hbox{ otherwise,}
                              \end{array}
                            \right. $$
and they form a basis for $\sum_{i\in I \setminus \{0\}} \Q \alpha_i$.

Let $\g$ be the affine Kac-Moody algebra associated with the Cartan datum $(\cmA, \wlP, \Pi, \Pi^{\vee})$ and let $U_q(\g)$ be its quantum group.
The quantum group $U_q(\g)$ is a $\C(q)$-algebra generated by $f_i$, $e_i$ $(i\in I)$ and $q^h$ $(h\in \wlP)$ with certain relations (see \cite[Chap.\ 3]{HK02}) for details).
Let $\A=\Z[q,q^{-1}]$. We denote by $U_\A^-(\g)$ the subalgebra of $U_q(\g)$ generated by $f_i^{(n)} := f_i^n / [n]_i!$ for $i\in I$ and $n\in \Z_{\ge0}$,
where  $q_i = q^{\mathsf{d}_i}$ and
\begin{equation*}
 \begin{aligned}
 \ &[n]_i =\frac{ q^n_{i} - q^{-n}_{i} }{ q_{i} - q^{-1}_{i} },
 \ &[n]_i! = \prod^{n}_{k=1} [k]_i.
 \end{aligned}
\end{equation*}

For a dominant integral weight $\Lambda \in \wlP^+$, let $V(\Lambda)$ be the irreducible highest weight $U_q(\g)$-module with highest weight $\Lambda$
and $V_\A(\Lambda)$ the $U_\A^-(\g)$-submodule of $V(\Lambda)$ generated by the highest weight vector.
As is usual, we denote by $B(\Lambda)$
the crystal associated with $V(\Lambda)$. We use standard notation $(\wt, \kf_i, \ke_i, \varepsilon_i, \varphi_i)$ ($i\in I$) for
crystal structure (see \cite[Chap.\ 4]{HK02} for details).

The Fock space representation for $U_q(C_\ell^{(1)})$ was constructed in \cite{KMM93} by folding the Fock space representation for $U_q(A_{2\ell-1}^{(1)})$ via the Dynkin diagram automorphism.
Later, the combinatorial description for the Fock space and its crystal base were developed in \cite{KimShin04, Pre04}.
Let us recall the combinatorial realization for $\Lambda_0$. 

Let $\lambda = (\lambda_1 \ge \lambda_2 \ge \cdots \ge \lambda_l >0 )$ be a Young diagram of size $|\lambda| := \sum_{i=1}^l \lambda_i$.
When $|\lambda|=n$, we write $\lambda\vdash n$. We consider the residue pattern
\begin{align} \label{Eq: residue seq}
0,1,2, \ldots , \ell-1, \ell, \ell-1, \ldots, 2, 1.
\end{align}
We repeat the residue pattern in the first row, and shift it to the right by one in the next row. It $b$ is a node of residue $i$ at the $(p,q)$-position, $b$ is called an $i$-node and $\res(p,q)=i$.
For example, when $\ell=4$ and $\lambda = (12,10,4,2)$, we have $\res(2,5) = 3$ and the residues are given as follows:
$$
\xy
(0,12)*{};(72,12)*{} **\dir{-};
(0,6)*{};(72,6)*{} **\dir{-};
(0,0)*{};(60,0)*{} **\dir{-};
(0,-6)*{};(24,-6)*{} **\dir{-};
(0,-12)*{};(12,-12)*{} **\dir{-};
(0,12)*{};(0,-12)*{} **\dir{-};
(6,12)*{};(6,-12)*{} **\dir{-};
(12,12)*{};(12,-12)*{} **\dir{-};
(18,12)*{};(18,-6)*{} **\dir{-};
(24,12)*{};(24,-6)*{} **\dir{-};
(30,12)*{};(30,0)*{} **\dir{-};
(36,12)*{};(36,0)*{} **\dir{-};
(42,12)*{};(42,0)*{} **\dir{-};
(48,12)*{};(48,0)*{} **\dir{-};
(54,12)*{};(54,0)*{} **\dir{-};
(60,12)*{};(60,0)*{} **\dir{-};
(66,12)*{};(66,6)*{} **\dir{-};
(72,12)*{};(72,6)*{} **\dir{-};
(3,9)*{0}; (9,9)*{1}; (15,9)*{2}; (21,9)*{3}; (27,9)*{4}; (33,9)*{3}; (39,9)*{2};(45,9)*{1};(51,9)*{0}; (57,9)*{1}; (63,9)*{2}; (69,9)*{3};
(3,3)*{1}; (9,3)*{0}; (15,3)*{1}; (21,3)*{2}; (27,3)*{3}; (33,3)*{4}; (39,3)*{3}; (45,3)*{2}; (51,3)*{1}; (57,3)*{0};
(3,-3)*{2}; (9,-3)*{1}; (15,-3)*{0}; (21,-3)*{1};
(3,-9)*{3}; (9,-9)*{2};
\endxy
$$

\bigskip
\noindent
Let $\ST(\lambda)$ be the set of all standard tableaux of shape $\lambda\vdash n$. For $T\in \ST(\lambda)$, we define the {\it residue sequence} of $T$ by
$$ \res(T) = (\res_1(T), \res_2(T), \ldots, \res_n(T) ), $$
where $\res_k(T)$ is the residue of the node of entry $k$ in $T$, for $1\le k \le n$.

Let $\lambda$ be a Young diagram.
 By an {\it addable} (resp.\ {\it removable}) node $b$ of $\lambda$, we mean a node which can be added to (resp.\ removed from) $\lambda$ to obtain
another Young diagram $\lambda \swarrow b$ (resp.\ $\lambda \nearrow b$).
For an addable or removable node $b$ with $\res(b)=i$, we set
\begin{align*}
d_b(\lambda) &:= \mathsf{d}_i \big( \# \{ \text{addable $i$-nodes of strictly below $b$} \} \\
& \qquad - \# \{ \text{removable $i$-nodes of strictly below $b$} \} \big), \\
d^b(\lambda) &:= \mathsf{d}_i \big( \# \{ \text{addable $i$-nodes of strictly above $b$} \} \\
& \qquad - \# \{ \text{removable $i$-nodes of strictly above $b$} \} \big), \\
d_i(\lambda) &:= \# \{ \text{addable $i$-nodes of $\lambda$} \} - \# \{ \text{removable $i$-nodes of $\lambda$} \},
\end{align*}
where $\mathsf{d}_i$ is given in $\eqref{Eq: d}$. Let $\F$ be the $\Q(q)$-vector space generated by all Young diagrams, which is the {\it Fock space} concerned in this paper.
For a Young diagram $\lambda \in \F$, we define
\begin{align} \label{Eq: actions of e,f}
e_i \lambda = \sum_{b} q^{d_b(\lambda)}\   \lambda \nearrow b, \quad f_i \lambda = \sum_{b} q^{-d^b(\lambda)}\   \lambda \swarrow b,
\end{align}
where $b$ runs over all removable $i$-nodes and all addable $i$-nodes respectively. Then,
the actions $e_i$ and $f_i$ give a $U_q(\g)$-module structure on $\F$, and we have $q^{h_i}\lambda=q^{d_i(\lambda)}\lambda$, for $i\in I$.

We identify the crystal basis of the Fock space with the set of all Young diagrams.
Its crystal structure can be described by considering the usual {\it $i$-signature}. Let $\lambda$ be a Young diagram, and
consider all addable or removable $i$-nodes $b_1, b_2, \ldots, b_m$ of $\lambda$ from top to bottom.
To each $b_k$ of $\lambda$, we assign its signature $s_k$ as $+$ (resp.\ $-$) if it is addable (resp.\ removable).
We cancel out all possible $(-,+)$ pairs in the $i$-signature $(s_1, \ldots, s_m  )$ so that a sequence of $+$'s is followed by $-$'s.
We define $\kf_i \lambda$ to be a Young diagram obtained from $\lambda$ by adding a node to the addable node corresponding to the right-most $+$ in the
$i$-signature. Similarly, $\ke_i \lambda$ is defined to be a Young diagram obtained from $\lambda$ by removing the removable node corresponding to the left-most $-$ in the
$i$-signature. Then, the Young diagrams form a $U_q(\g)$-crystal.

We remark that the above description is obtained from the description in \cite[Thm. 1.3]{KimShin04} and \cite[Thm.3.1]{Pre04} by flipping Young diagrams diagonally.
This description matches with the description of the Fock space for affine type $A$ given in \cite[Sec.\ 3.6]{BK2-09}.

\vskip 1em

\section{Quiver Hecke algebras}\label{Sec: quiver Hecke algs}

Let $\bR$ be an algebraically closed field and $(\cmA, \wlP, \Pi, \Pi^{\vee})$ the affine Cartan datum in Section \ref{Sec: Fock space}. We set polynomials $\mathcal{Q}_{i,j}(u,v)\in\bR[u,v]$, for $i,j\in I$,
of the form
\begin{align*}
\mathcal{Q}_{i,j}(u,v) = \left\{
                 \begin{array}{ll}
                   \sum_{p(\alpha_i|\alpha_i)+q (\alpha_j|\alpha_j) + 2(\alpha_i|\alpha_j)=0} t_{i,j;p,q} u^pv^q & \hbox{if } i \ne j,\\
                   0 & \hbox{if } i=j,
                 \end{array}
               \right.
\end{align*}
where $t_{i,j;p,q} \in \bR$ are such that $t_{i,j;-a_{ij},0} \ne 0$ and $\mathcal{Q}_{i,j}(u,v) = \mathcal{Q}_{j,i}(v,u)$.
The symmetric group $\mathfrak{S}_n = \langle s_k \mid k=1, \ldots, n-1 \rangle$ acts on $I^n$ by place permutations.

\begin{defn} \
The {\it quiver Hecke algebra} $R(n)$ associated with polynomials $(\mathcal{Q}_{i,j}(u,v))_{i,j\in I}$
is the $\Z$-graded  $\bR$-algebra defined by three sets of generators
$$\{e(\nu) \mid \nu = (\nu_1,\ldots, \nu_n) \in I^n\}, \;\{x_k \mid 1 \le k \le n\}, \;\{\psi_l \mid 1 \le l \le n-1\} $$
subject to the following relations:

\begin{align*}
& e(\nu) e(\nu') = \delta_{\nu,\nu'} e(\nu),\ \sum_{\nu \in I^{n}} e(\nu)=1,\
x_k e(\nu) =  e(\nu) x_k, \  x_k x_l = x_l x_k,\\
& \psi_l e(\nu) = e(s_l(\nu)) \psi_l,\  \psi_k \psi_l = \psi_l \psi_k \text{ if } |k - l| > 1, \\[5pt]
&  \psi_k^2 e(\nu) = \mathcal{Q}_{\nu_k, \nu_{k+1}}(x_k, x_{k+1}) e(\nu), \\[5pt]
&  (\psi_k x_l - x_{s_k(l)} \psi_k ) e(\nu) = \left\{
                                                           \begin{array}{ll}
                                                             -  e(\nu) & \hbox{if } l=k \text{ and } \nu_k = \nu_{k+1}, \\
                                                               e(\nu) & \hbox{if } l = k+1 \text{ and } \nu_k = \nu_{k+1},  \\
                                                             0 & \hbox{otherwise,}
                                                           \end{array}
                                                         \right. \\[5pt]
&( \psi_{k+1} \psi_{k} \psi_{k+1} - \psi_{k} \psi_{k+1} \psi_{k} )  e(\nu) \\[4pt]
&\qquad \qquad \qquad = \left\{
                                                                                   \begin{array}{ll}
\displaystyle \frac{\mathcal{Q}_{\nu_k,\nu_{k+1}}(x_k,x_{k+1}) -
\mathcal{Q}_{\nu_k,\nu_{k+1}}(x_{k+2},x_{k+1})}{x_{k}-x_{k+2}} e(\nu) & \hbox{if } \nu_k = \nu_{k+2}, \\
0 & \hbox{otherwise}. \end{array}
\right.\\[5pt]
\end{align*}
\end{defn}

Using the isomorphism given in \cite[page 25]{R08} (cf.\ \cite[Lemma 2.2]{AIP13}), we may assume that, for $i< j$,
\begin{align*}
\mathcal{Q}_{i,j}(u,v)= \left\{
                          \begin{array}{ll}
                            u + v^2 & \hbox{ if } i = 0, j=1, \\
                            u+v & \hbox{ if } j = i+1, i\ne 0, j \ne \ell, \\
                            u^2 + v & \hbox{ if } i= \ell-1, j = \ell, \\
                            1 & \hbox{ otherwise.}
                          \end{array}
                        \right.
\end{align*}

$R(n)$ is a graded algebra by the $\Z$-grading given as follows:
\begin{align*}
\deg(e(\nu))=0, \quad \deg(x_k e(\nu))= ( \alpha_{\nu_k} |\alpha_{\nu_k}), \quad  \deg(\psi_l e(\nu))= -(\alpha_{\nu_{l}} | \alpha_{\nu_{l+1}}).
\end{align*}

For an $R(m)$-module $M$ and an $R(n)$-module $N$, we define an $R(m+n)$-module $M \circ N$ by
$$ M \circ N = R(m+n) \otimes_{R(m) \otimes R(n)} (M \otimes N). $$

For a dominant integral weight $\Lambda \in \wlP^+$, let $R^\Lambda(n) $ be the quotient algebra of $R(n)$ by the ideal generated
by the elements $ \{x_1^{\langle h_{\nu_1}, \Lambda \rangle} e(\nu) \mid \nu\in I^n\}$, which is called the {\it cyclotomic quiver Hecke algebra}.

For $\beta\in \rlQ^+$ with $|\beta|=n$, we set $I^\beta = \{ \nu=(\nu_1, \ldots, \nu_n ) \in I^n \mid \sum_{k=1}^n\alpha_{\nu_k} = \beta \}$ and define
$$ R^\Lambda(\beta) := R^\Lambda(n) e(\beta), $$
where $ e(\beta) = \sum_{\nu \in I^\beta} e(\nu)$. We are interested in cyclotomic quiver Hecke algebras $\fqH(\beta)$, which we call \emph{finite quiver Hecke algebras of type $C_{\ell}^{(1)}$}.
Let us recall some results which are valid for general $R^\Lambda(\beta)$.

\begin{prop} [\protect{cf.\ \cite[Cor.\ 4.8]{AP12}}] \label{Prop: repn type}
For $w \in \weyl$, $R^\Lambda(\beta)$ and $R^\Lambda(\Lambda - w\Lambda + w\beta )$
have the same number of simple modules and the same representation type.
\end{prop}

\noindent
We denote the direct sum of the split Grothendieck groups of the categories $R^\Lambda (\beta)\text{-}\proj$ of finitely generated projective graded $R^\Lambda(\beta)$-modules by
$$ K_0(R^\Lambda) = \bigoplus_{\beta\in \rlQ^+} K_0(R^\Lambda (\beta)\text{-}\proj). $$
Note that $K_0(R^\Lambda)$ has a free $\A$-module structure
induced from the $\Z$-grading on $R^\Lambda(\beta)$, i.e.\ $(qM)_k =
M_{k-1}$ for a graded module $M = \bigoplus_{k \in \Z} M_k $. Let
$e(\nu, i)$ be the idempotent corresponding to the concatenation of
$\nu$ and $(i)$, and set $e(\beta, i) = \sum_{\nu \in I^\beta}
e(\nu, i)$ for $\beta \in \rlQ^+$. Then we define the induction functor $F_i: R^\Lambda(\beta)\mod \rightarrow R^\Lambda(\beta+\alpha_i)\mod$ and
the restriction functor $E_i: R^\Lambda(\beta+\alpha_i)\mod \rightarrow R^\Lambda(\beta)\mod$ by
$$ F_i(M) = R^\Lambda(\beta+ \alpha_i) e(\beta,i) \otimes_{R^\Lambda(\beta)}M, \qquad \ E_i(N)  =  e(\beta,i)N ,$$
for an $R^\Lambda(\beta)$-module $M$ and an $R^\Lambda(\beta+ \alpha_i)$-module $N$.
\begin{thm} [\protect{\cite[Thm.\ 5.2]{KK11}}] \label{Thm: Ei Fi}
Let $l_i = \langle h_i, \Lambda - \beta  \rangle$, for $i\in I$. Then one of the following  isomorphisms of endofunctors
on $R^\Lambda(\beta)\mod$ holds.
\begin{enumerate}
\item If $l_i \ge 0$, then
$$ q_i^{-2}F_i E_i \oplus \bigoplus_{k=0}^{l_i- 1} q_i^{2k} \mathrm{id} \buildrel \sim \over \longrightarrow E_iF_i .$$
\item If $l_i \le 0$, then
$$ q_i^{-2}F_i E_i  \buildrel \sim \over \longrightarrow  E_iF_i \oplus \bigoplus_{k=0}^{-l_i- 1} q_i^{-2k-2}  \mathrm{id} .$$
\end{enumerate}
\end{thm}

\noindent
Moreover, the functor $q_i^{1 - \langle h_i, \Lambda - \beta \rangle}E_i$ and $F_i$ give a
$U_\A(\g)$-module structure to $K_0(R^\Lambda)$.
\begin{thm}[\protect{\cite[Thm.\ 6.2]{KK11}}] \label{Thm: categorification thm}
There exists a $U_\A(\g)$-module isomorphism between $K_0(R^\Lambda)$ and $V_\A(\Lambda)$.
\end{thm}

For a graded module $ M = \bigoplus_{k\in \Z} M_k $, the {\it graded dimension} of $M$ is defined by
$$ \dim_q M = \sum_{k\in \Z} \dim(M_k)q^{k} . $$
Note that $\dim_q (q^t M) = q^t \dim_q M$. For an $R^\Lambda(\beta)$-module $M$, the {\it $q$-character} $\mathrm{ch}_q(M)$ and {\it character} $\mathrm{ch}(M)$ of $M$
are defined by
$$
\mathrm{ch}_q(M) = \sum_{\nu \in I^\beta} \dim_q(e(\nu)M)\nu, \quad \mathrm{ch}(M) = \sum_{\nu \in I^\beta} \dim(e(\nu)M)\nu.
$$
For $\Lambda \in \wlP^+$ and $\beta \in \rlQ^+$, set
\begin{align*}
\mathsf{def} (\Lambda, \beta) = (\beta | \Lambda) - \frac{1}{2} (\beta|\beta).
\end{align*}
Using $(\alpha_i | \alpha_i )=2\mathsf{d}_i$, it is easy to check
$$ \mathsf{def} (\Lambda, \beta-\alpha_i) + ( \Lambda-\beta | \alpha_i ) = \mathsf{def} (\Lambda, \beta) -\mathsf{d}_i. $$

\begin{prop} [\protect{\cite[Prop.\ 2.3]{OP14}}]  \label{Prop: def and dim}
Let $\nu = (\nu_1, \ldots, \nu_n), \nu' = (\nu_1', \ldots, \nu_n') \in I^\beta$, and let $v_{\Lambda}$ be the highest weight vector of the highest weight $U_q(\g)$-module $V(\Lambda)$. Then, we have
$$ e_{\nu_1} \cdots e_{\nu_n} f_{\nu_n'} \cdots f_{\nu_1'} v_{\Lambda} = q^{- \mathsf{def} (\Lambda, \beta) } \left(\dim_q e(\nu)R^{\Lambda}(\beta) e(\nu') \right) v_{\Lambda}. $$
\end{prop}

We now consider the $q$-dimension $\dim_q \fqH(\beta)$. Let $\lambda \vdash n$ and $T$ be a standard tableau of shape $\lambda$.
For $1 \le k \le n$, let $T_{<k}$ be a standard tableau obtained from $T$ by removing the nodes whose entries are greater than or equal to $k$.
We define inductively
$$ \deg(T) := \deg(T_{< n}) + d_b(\lambda), \quad
   \codeg(T) := \codeg(T_{< n}) + d^b(\lambda \nearrow b), $$
where $b$ is the node of $T$ containing entry $n$. We set $\deg(\emptyset) = \codeg(\emptyset) = 0$.
Observe that if $b$ is a removable $i$-node, then
$$ d_b(\lambda)+d^b(\lambda \nearrow b)=\mathsf{d}_id_i(\lambda)+\mathsf{d}_i. $$
One can prove the following identity by the same induction argument as \cite[Lem.\ 3.12]{BKW11}:
\begin{align} \label{Eq: def}
\deg(T) + \codeg(T) = \mathsf{def} (\Lambda_0, \beta).
\end{align}
For $\nu \in I^n$, let
$$ K_q(\lambda, \nu) := \sum_{T \in \ST(\lambda),\ \res(T)= \nu} q^{\deg(T)}, \qquad K_q(\lambda) := \sum_{T \in \ST(\lambda)} q^{\deg(T)}. $$

\begin{thm} \label{Thm: dimension forumula}
For $\nu, \nu' \in I^\beta$, we have
\begin{align*}
\dim_q e(\nu) \fqH(\beta) e(\nu') &= \sum_{\lambda \vdash n,\ \wt(\lambda) = \Lambda_0 - \beta}  K_q(\lambda, \nu) K_q(\lambda, \nu') ,\\
\dim_q \fqH(\beta) &= \sum_{\lambda \vdash n,\ \wt(\lambda) = \Lambda_0 - \beta}  K_q(\lambda)^2, \\
\dim_q \fqH(n) &= \sum_{\lambda \vdash n}  K_q(\lambda)^2.
\end{align*}
\end{thm}
\begin{proof} Let $\nu = (\nu_1, \ldots, \nu_n)$ and $\nu' = (\nu_1', \ldots, \nu_n') \in I^\beta$. It follows from $\eqref{Eq: actions of e,f}$ and $\eqref{Eq: def}$ that
\begin{align*}
q^{\mathsf{def}(\Lambda_0, \beta)} &  e_{\nu_1} \cdots e_{\nu_n} f_{\nu_n'} \cdots f_{\nu_1'} \emptyset  \\
&=q^{\mathsf{def}(\Lambda_0, \beta)}  \sum_{ \lambda \vdash n,\ \wt(\lambda)= \Lambda_0 - \beta}
\left(   \sum_{ \substack{T \in \ST(\lambda),\\ \res(T) = \nu}} q^{\deg(T)} \right)
\left(   \sum_{\substack{T \in \ST(\lambda),\\ \res(T) = \nu}} q^{-\codeg(T)} \right) \emptyset \\
&= \sum_{\lambda \vdash n,\ \wt(\lambda) = \Lambda_0 - \beta}  K_q(\lambda, \nu) K_q(\lambda, \nu') \emptyset,
\end{align*}
which gives the first assertion by Proposition \ref{Prop: def and dim}.

The remaining assertions follow from $  \fqH(\beta) = \bigoplus_{\nu, \nu' \in I^\beta} e(\nu) \fqH(\beta) e(\nu')$ and $ \fqH(n) = \bigoplus_{ |\beta| = n} \fqH(\beta)$.
\end{proof}

The corollary below follows from Theorem \ref{Thm: dimension forumula} immediately.
\begin{cor} \label{Cor: dimension formula}
\begin{enumerate}
\item Let $\nu \in I^n$. Then, $e(\nu) \ne 0 $ in $\fqH(n)$ if and only if $\nu$ may be obtained from a standard tableau $T$ as $\nu = \res(T)$.
\item For a natural number $n$, we have $ \dim \fqH(n) = n!. $
\end{enumerate}
\end{cor}

\vskip 1em

\section{Representations of $\fqH(\delta)$}

In \cite[Sec.\ 8.1]{KM13}, irreducible $\fqH(\delta)$-modules for several non-simply laced affine types were constructed. Let us recall the construction for type $C_\ell^{(1)}$.

Let $z$ be an indeterminate. For $k=0,1,2,3$ and $1\le i\le \ell$, except for $(k,i)=(2,1)$, let $L^z_{i,k}$ be the graded free 1-dimensional $\bR[z]$-module with generator $v_k$, and set
$$ \nu^{(k)} = \left\{
                 \begin{array}{ll}
                   (0) & \hbox{ if } k =0, \\
                   (1,2,\ldots, \ell-1,\ell, \ell-1, \ldots, i+1) & \hbox{ if } k=1,\ 1\le i < \ell, \\
                   (1,2,\ldots, \ell-1) & \hbox{ if } k=1,\ i = \ell, \\
                   (1,2,\ldots, i-1) & \hbox{ if } k=2,\ 2\le i\le\ell \\
                   (i) & \hbox{ if } k =3.
                 \end{array}
               \right.
 $$
We set $\beta^{(k)} = \alpha_{\nu_1} + \cdots + \alpha_{\nu_t} $, where $\nu^{(k)} = (\nu_1, \nu_2, \ldots, \nu_t )$.
Define an $\fqH(\beta^{(k)})$-module structure on $L^z_{i,k}$ by $e(\nu) v_k = \delta_{\nu, \nu^{(k)} } v_k$, $\psi_r v_k = 0$ and
\begin{align*}
x_s v_k = \left\{
            \begin{array}{ll}
              z v_k & \hbox{ if } k=1,\ s < \ell, \\
              -z v_k & \hbox{ if } (k=1,\ s > \ell) \text{ or } (k=2) \text{ or } (k=3,\ i < \ell), \\
              z^2 v_k & \hbox{ if } (k=0) \text{ or }  (k=1,\ s=\ell) \text{ or } (k=3,\ i=\ell).
            \end{array}
          \right.
\end{align*}

We set
\begin{align} \label{Eq: Lzi}
L^z_i = \left\{
          \begin{array}{ll}
           L^z_{i,0} \boxtimes L^z_{i,1} \boxtimes L^z_{i,3}  & \hbox{ if } i=1, \\
           L^z_{i,0} \boxtimes L^z_{i,1} \circ L^z_{i,2} \boxtimes L^z_{i,3} & \hbox{ if } i > 1,
          \end{array}
        \right.
\end{align}
and declare that $\psi_1$ and $\psi_{2\ell-1}$ act as $0$ on $L^z_i$.

Note that the choice of the polynomials ${\mathcal Q}_{i,j}(u,v)$ does not affect the modules $L^z_{i,k}$. Thus, we may
use the following result without any change.
\begin{prop}[\protect{\cite[Prop.\ 3.9.2,\ Prop.\ 8.1.3,\ Prop.\ 8.1.6]{KM13}}] \label{Prop: simples for delta}\
\begin{enumerate}
\item[(1)] For $i=1, \ldots, \ell$, $L^z_i$ is a $\bR[z] \otimes R(\delta)$-module.
\item[(2)] The quotient $ \mathcal{S}_i := L^z_i / z L^z_i $ is an irreducible $\fqH(\delta)$-module.
\item[(3)] $\{ \mathcal{S}_1, \mathcal{S}_2 \ldots , \mathcal{S}_\ell \}$ is a complete list of irreducible $\fqH(\delta)$-modules.
\end{enumerate}
\end{prop}

\begin{lemma}\label{irreducibility}
If $M$ is an irreducible $R(\beta)$-module with $\varepsilon_i(M)=1$ then $E_iM$ is an irreducible $R(\beta-\alpha_i)$-module.
\end{lemma}
\begin{proof}
It immediately follows from \cite[Lem. 3.8]{KL09}.
\end{proof}

By the definition of $\mathcal{S}_i$, we may enumerate basis elements of $L^z_{i,1} \circ L^z_{i,2}$ and we have the following
description of the characters for $\mathcal{S}_i$.
\begin{align} \label{Eq: char of S}
\mathrm{ch} \mathcal{S}_i = \sum_{T \in \ST(\lambda^{(i)})} \res(T)*(i) ,
\end{align}
where $\lambda^{(i)} = ( i, 1^{2\ell-1-i})$ and $\res(T)*(i)$ is the concatenation of $\res(T)$ and $(i)$. Thus,
we have $\varepsilon_j( \mathcal{S}_i ) = \delta_{i,j}$, and Lemma \ref{irreducibility} implies that
$$
\mathcal{L}_i := E_i \mathcal{S}_i
$$
is an irreducible $\fqH(\delta - \alpha_i)$-module, for $i=1,\cdots,\ell$.
Using $\eqref{Eq: char of S}$ again, if $i\ne\ell$ then
\begin{align} \label{Eq: epsilon2}
\varepsilon_j(\mathcal{L}_i ) = \left\{
                        \begin{array}{ll}
                          1 & \hbox{ if } j = i+1, i-1, \\
                          0 & \hbox{ otherwise.}
                        \end{array}
                      \right.
\end{align}
Thus, $E_{i\pm 1} \mathcal{L}_i $ is an irreducible $\fqH(\delta - \alpha_i - \alpha_{i\pm1})$-module, for $1\le i\le \ell-1$,
by Lemma \ref{irreducibility}.

\begin{lemma} \label{Lem: alg type}
\begin{enumerate}
\item[(1)] $\fqH(\alpha_0 + \alpha_1)$ is isomorphic to $\bR[x]/(x^2)$.
\item[(2)] For $1\le i \le \ell-1$, $\fqH(\delta - \alpha_i)$ is isomorphic to a matrix ring over $\bR[x]/(x^2)$, and
$\mathcal{L}_i$ is the unique irreducible  $\fqH(\delta - \alpha_i)$-module.
\item[(3)] For $1\le i \le \ell-1$, $\fqH(\delta - \alpha_i - \alpha_{i +1})$ is isomorphic to a matrix ring over $\bR[x]/(x^2)$, and
$E_{i+1}\mathcal{L}_i\simeq E_i\mathcal{L}_{i+1}$ is the unique irreducible  $\fqH(\delta - \alpha_i - \alpha_{i +1})$-module
if $1\le i\le \ell-2$, and $E_\ell\mathcal{L}_{\ell-1}$ is the unique irreducible  $\fqH(\delta - \alpha_{\ell-1} - \alpha_\ell)$-module.
\item[(4)] $\fqH(\delta - \alpha_\ell)$ is a simple algebra and $\mathcal{L}_\ell$ is the unique irreducible $\fqH(\delta - \alpha_\ell)$-module.
\end{enumerate}
\end{lemma}
\begin{proof}
The assertion (1) follows from Theorem \ref{Thm: dimension forumula}.
Indeed, $\dim_q \fqH(2)=1+q^2$ implies that there is a homogeneous element $x\ne0$ of degree $2$ such that $x^2=0$.
One can verify the following formulas, for $p = I \setminus\{0, \ell-1, \ell \}$ and $t = I \setminus\{0, \ell \}$, by direct computation.
\begin{align*}
\Lambda_0 - \delta + \alpha_p + \alpha_{p+1} &= (r_{p-1}r_{p-2} \cdots r_{1})( r_{p+2} \cdots r_{\ell-1}r_{\ell}r_{\ell-1} \cdots r_3r_2) (\Lambda_0 - \alpha_0 - \alpha_1),\\
\Lambda_0 - \delta + \alpha_{\ell-1} + \alpha_{\ell} &= (r_{\ell-2}r_{\ell-3} \cdots r_{1})( r_{\ell-1} \cdots r_3r_2) (\Lambda_0 - \alpha_0 - \alpha_1),\\
\Lambda_0 - \delta + \alpha_t &= (r_{t-1}r_{t-2} \cdots r_{1})( r_{t+1} \cdots r_{\ell-1}r_{\ell}r_{\ell-1} \cdots r_3r_2) (\Lambda_0 - \alpha_0 - \alpha_1),\\
\Lambda_0 - \delta + \alpha_\ell &= r_{\ell-1} \cdots r_{2}r_{1}(\Lambda_0-\alpha_0).
\end{align*}

By \cite[Thm.\ 6.4]{CR08}(cf.\ \cite[Thm.\ 4.5]{AP12}),
$\fqH(\delta - \alpha_i)$ and $\fqH(\delta - \alpha_i - \alpha_{i\pm 1})$ are derived equivalent to $\fqH(\alpha_0 + \alpha_1)$.
Since $\bR[x]/(x^2)$ is the unique Brauer tree algebra with one edge and no exceptional vertex, both
$\fqH(\delta - \alpha_i)$ and $\fqH(\delta - \alpha_i - \alpha_{i\pm 1})$ are Morita equivalent to $\bR[x]/(x^2)$
by \cite[Thm.\ 4.2]{Ri89}. In particular, they have a unique irreducible module.
As we already know that $\mathcal{L}_i$ is an irreducible $\fqH(\delta-\alpha_i)$-module,
(2) follows. We also know that $E_{i+1}\mathcal{L}_i$, for $1\le i \le \ell-1$, and $E_i\mathcal{L}_{i+1}$,
for $1\le i\le \ell-2$, are irreducible $\fqH(\delta - \alpha_i - \alpha_{i +1})$-modules. Thus (3) follows.
Finally, Proposition \ref{Prop: repn type} tells that $\fqH(\delta - \alpha_\ell)$ is a simple algebra, and
we already know that $\mathcal{L}_\ell$ is an irreducible $\fqH(\delta - \alpha_\ell)$-module, which proves (4).
\end{proof}

By Lemma \ref{Lem: alg type}(4), $\mathcal{L}_\ell$ is a projective module. For $i \ne \ell$, we denote
the projective cover of $\mathcal{L}_i$ by $\widehat{\mathcal{L}}_i$. Then, we have a non-split exact sequence
\begin{align} \label{Eq: widehat L}
0 \rightarrow \mathcal{L}_i \rightarrow \widehat{\mathcal{L}}_i \rightarrow \mathcal{L}_i \rightarrow 0 .
\end{align}
We get indecomposable projective $\fqH(\delta-\alpha_i)$-modules $\mathcal{M}_i$, for $1\le i\le \ell$, defined by
$$\mathcal{M}_i := \left\{
                     \begin{array}{ll}
                        \widehat{\mathcal{L}}_i  & \hbox{ if } i \ne \ell,  \\
                        \mathcal{L}_\ell & \hbox{ if } i = \ell.
                     \end{array}
                   \right.
$$

\begin{lemma}\label{PIM isom}
We have $E_j\mathcal{M}_i=0$ unless $j=i\pm1$. If $j=i\pm1$ then $E_i\mathcal{M}_j\simeq E_j\mathcal{M}_i$ is the
unique indecomposable projective $\fqH(\delta-\alpha_i-\alpha_j)$-module.
\end{lemma}
\begin{proof}
If $j\ne i\pm1$, then $E_j\mathcal{M}_i=0$ follows from (\ref{Eq: epsilon2}). Computation of the characters implies
$[E_{\ell-1}\mathcal{L}_\ell]=2[E_\ell\mathcal{L}_{\ell-1}]$, which is equal to $[E_\ell\widehat{\mathcal{L}}_{\ell-1}]$.
Since $E_{\ell-1}\mathcal{M}_\ell$ and $E_\ell\mathcal{M}_{\ell-1}$ are projective modules,
$[E_{\ell-1}\mathcal{M}_\ell]=[E_\ell\mathcal{M}_{\ell-1}]$ implies that they are isomorphic. Suppose that
$i\ne\ell$, $j\ne\ell$ and $j=i\pm1$. Then we have the exact sequence
\begin{align} \label{Eq: E Mi}
0 \rightarrow E_j\mathcal{L}_i \rightarrow E_j \mathcal{M}_i \rightarrow E_j \mathcal{L}_i \rightarrow 0 .
\end{align}
If $E_j\mathcal{L}_i$ was a projective module, it would contradict Lemma \ref{Lem: alg type}(3). Thus,
$E_j\mathcal{L}_i$ is not projective and (\ref{Eq: E Mi}) does not split. It implies that $E_j \mathcal{M}_i$ is
an indecomposable projective $\fqH(\delta-\alpha_i-\alpha_j)$-module. Interchanging the role of $i$ and $j$,
$E_i \mathcal{M}_j$ is also an indecomposable projective $\fqH(\delta-\alpha_i-\alpha_j)$-module. As the indecomposable
projective $\fqH(\delta-\alpha_i-\alpha_j)$-module is unique by Lemma \ref{Lem: alg type}(3), we conclude that
they are isomorphic.
\end{proof}

We now consider the projective $\fqH(\delta)$-modules $\mathcal{P}_i := F_i \mathcal{M}_i$, for $1\le i \le \ell$.
By the biadjointness of $F_i$ and $E_i$ \cite{Kash11} and $\varepsilon_j(\mathcal{S}_i)=\delta_{i,j}$, we have
\begin{align*}
\dim\Hom (\mathcal{P}_i, \mathcal{S}_j) &= \dim \Hom (\mathcal{M}_i, E_i\mathcal{S}_j) = \delta_{i,j}\dim \Hom (\mathcal{M}_i, \mathcal{L}_i) = \delta_{i,j},    \\
\dim \Hom (\mathcal{S}_j, \mathcal{P}_i) &= \dim \Hom (E_i\mathcal{S}_j, \mathcal{M}_i) = \delta_{i,j}\dim \Hom (\mathcal{L}_i, \mathcal{M}_i) = \delta_{i,j},
\end{align*}
which tells that $\mathcal{P}_i$ is the projective cover of $\mathcal{S}_i$, for all $i$, and $\fqH(\delta)$ is weakly symmetric.
In particular, $\mathcal{P}_i$ are self-dual.
It follows from Theorem \ref{Thm: Ei Fi} and Lemma \ref{PIM isom} that, if $i \ne j$ then
\begin{align*}
\dim\Hom (\mathcal{P}_j, \mathcal{P}_i) = \dim\Hom (\mathcal{M}_j, E_jF_i\mathcal{M}_i)
= \dim\Hom (E_i\mathcal{M}_j, E_j\mathcal{M}_i) = 2 \delta_{j, i \pm 1}.
\end{align*}
The similar argument shows that
\begin{align*}
\dim\Hom (\mathcal{P}_i, \mathcal{P}_i) &= \dim\Hom (\mathcal{M}_i, E_iF_i\mathcal{M}_i) \\
&= \dim\Hom (\mathcal{M}_i, \mathcal{M}_i^{\oplus \langle h_i, \Lambda_0 - \delta + \alpha_i \rangle }) \\
&=
\left\{
  \begin{array}{ll}
    4 & \hbox{ if } i \ne \ell, \\
    2 & \hbox{ if } i = \ell.
  \end{array}
\right.
\end{align*}
Thus, in the Grothendieck group, we have
\begin{align} \label{Eq: composition factors}
[\mathcal{P}_1] = 4[\mathcal{S}_1]  + 2[\mathcal{S}_2], \quad
[\mathcal{P}_i] = 2[\mathcal{S}_{i-1}] + 4[\mathcal{S}_i]  + 2[\mathcal{S}_{i+1}], \quad
[\mathcal{P}_\ell] = 2[\mathcal{S}_{\ell-1}] + 2[\mathcal{S}_\ell],
\end{align}
for $i = 2, \ldots, \ell-1$.

Define $\mathcal{Q}_i := F_i \mathcal{L}_i$, for $i \ne \ell$.
By the same argument as above, we compute
\begin{align*}
&\dim\Hom (\mathcal{Q}_i, \mathcal{S}_j) = \dim\Hom (\mathcal{S}_j, \mathcal{Q}_i) = \delta_{i,j}.
\end{align*}
Applying the functor $F_i$ to $\eqref{Eq: widehat L}$, and noting that $\mathcal{P}_i$ is indecomposable,
we have the following non-split exact sequence, for $i=1, \ldots, \ell-1$.
\begin{align} \label{Eq: P and Q}
0 \rightarrow \mathcal{Q}_i \rightarrow \mathcal{P}_i \rightarrow \mathcal{Q}_i \rightarrow 0 .
\end{align}
Since $\mathcal{P}_i$ is self-dual, and $\Soc(\mathcal{Q}_i)\simeq \mathcal{S}_i\simeq \Top(\mathcal{Q}_i)$, we conclude that
\begin{align} \label{Eq: Radicals for Q}
\mathcal{Q}_1 \simeq \begin{array}{c}
                       \mathcal{S}_1 \\
                       \mathcal{S}_2 \\
                       \mathcal{S}_1
                     \end{array},
\qquad \qquad
\mathcal{Q}_i \simeq \begin{array}{c}
                       \mathcal{S}_i \\
                       \mathcal{S}_{i-1} \oplus \mathcal{S}_{i+1} \\
                       \mathcal{S}_i
                     \end{array}
\quad (2 \le i \le \ell-1).
\end{align}
The radical series for $\mathcal{Q}_1$ is clear. Suppose that $\mathcal{Q}_i$, for some $2\le i\le \ell-1$ is uniserial.
If $\Rad(\mathcal{Q}_i)/\Rad^2(\mathcal{Q}_i)\simeq\mathcal{S}_{i\pm1}$ then $\mathcal{S}_{i\pm1}$ appears in
$\Rad(\mathcal{P}_i)/\Rad^2(\mathcal{P}_i)$ and $\mathcal{S}_{i\mp1}$ appears in $\Soc^2(\mathcal{P}_i)/\Soc(\mathcal{P}_i)$,
which implies that $\mathcal{S}_{i\pm1}\oplus\mathcal{S}_{i\mp1}$ appears in $\Rad(\mathcal{P}_i)/\Rad^2(\mathcal{P}_i)$. On the other hand, either
$2[\mathcal{S}_{i\pm1}]$ or $2[\mathcal{S}_{i\mp1}]$ all appear in $\Rad^2(\mathcal{P}_i)$. They contradict and we conclude that
$\mathcal{Q}_i$ is not uniserial. We have the desired shape of the radical series for $\mathcal{Q}_i$.

\begin{prop} \label{Prop: radical series for delta}
The radical series of $\mathcal{P}_i$, for $1\le i \le \ell$, are given as follows.
\begin{align*}
\mathcal{P}_1 \simeq \begin{array}{c}
                       \mathcal{S}_1 \\
                       \mathcal{S}_1 \oplus \mathcal{S}_2 \\
                       \mathcal{S}_2 \oplus \mathcal{S}_1 \\
                       \mathcal{S}_1
                     \end{array}
,\qquad
\mathcal{P}_i \simeq \begin{array}{c}
                       \mathcal{S}_i \\
                       \mathcal{S}_{i} \oplus \mathcal{S}_{i-1} \oplus \mathcal{S}_{i+1} \\
                       \mathcal{S}_{i+1} \oplus \mathcal{S}_{i-1} \oplus \mathcal{S}_{i} \\
                       \mathcal{S}_i
                     \end{array}
\quad(i\ne1,\ell),\qquad
\mathcal{P}_\ell \simeq \begin{array}{c}
                       \mathcal{S}_\ell \\
                       \mathcal{S}_{\ell-1}  \\
                       \mathcal{S}_{\ell-1}  \\
                       \mathcal{S}_\ell
                     \end{array}
\end{align*}
\end{prop}
\begin{proof}
We set $\widehat{\mathcal{S}}_1 := L^z_1 / z^2 L^z_1$, where $L^z_i$ is given in $\eqref{Eq: Lzi}$.
By definition, $x_1$ acts as zero, and $\widehat{\mathcal{S}}_1$ is an $\fqH(\delta)$-module.
On the other hand, $x_2$ acts as nonzero
on $\widehat{\mathcal{S}}_1$ by $\ell\ge2$. It implies that $\widehat{\mathcal{S}}_1$ is indecomposable and we have the radical series
$$ \widehat{\mathcal{S}}_1  \simeq \begin{array}{c}
                                     \mathcal{S}_1 \\
                                     \mathcal{S}_1
                                   \end{array} .
$$
Thus, $\Rad(\mathcal{P}_1)/ \Rad^2(\mathcal{P}_1)$ has $\mathcal{S}_1$ as a direct summand.
It follows from $\eqref{Eq: P and Q}$ and $\eqref{Eq: Radicals for Q}$ that $\mathcal{P}_1$ has the radical series as follows.
$$
\mathcal{P}_1 \simeq \begin{array}{c}
                       \mathcal{S}_1 \\
                       \mathcal{S}_1 \oplus \mathcal{S}_2 \\
                       \mathcal{S}_2 \oplus \mathcal{S}_1 \\
                       \mathcal{S}_1
                     \end{array}.
$$

Let $\phi:\mathcal{P}_2 \rightarrow \mathcal{P}_1$ be a lift of the map $\mathcal{P}_2 \twoheadrightarrow \mathcal{S}_2 \hookrightarrow \Rad(\mathcal{P}_1)/ \Rad^2(\mathcal{P}_1)$.
From the shape of the radical series of $\mathcal{P}_1$, we know that $\Rad^2(\Im \phi)\simeq \mathcal{S}_1$. It implies that
$\mathcal{S}_1$ appears in $\Rad^2(\mathcal{P}_2)/ \Rad^3(\mathcal{P}_2)$.
Under the projection $p_2:\mathcal{P}_2\to\mathcal{Q}_2$, this $\mathcal{S}_1$ maps to zero. Namely, it appears in $\Ker(p_2)\simeq \mathcal{Q}_2$.
Multiplying $\Rad(\fqH(\delta))$ to this $\mathcal{S}_1$, we know that $\Soc(\mathcal{P}_2)=\Rad^3(\mathcal{P}_2)$.
By $\eqref{Eq: P and Q}$ and $\eqref{Eq: Radicals for Q}$, $\mathcal{S}_2$ appears in $\Rad^2(\mathcal{P}_2)/ \Rad^3(\mathcal{P}_2)$.
It follows that $\mathcal{P}_2$ has a uniseirial submodule of length $2$ with two $\mathcal{S}_2$ as composition factors.
Hence, $\mathcal{S}_2$ appears in $\Rad(\mathcal{P}_2)/ \Rad^2(\mathcal{P}_2)$. Then, this $\mathcal{S}_2$ must appear in $\Ker(p_2)$,
which implies that $\mathcal{S}_3$ appears in $\Rad^2(\mathcal{P}_2)/\Rad^3(\mathcal{P}_2)$.
We conclude that
$$
\mathcal{P}_2 \simeq \begin{array}{c}
                       \mathcal{S}_2 \\
                       \mathcal{S}_{2} \oplus \mathcal{S}_{1} \oplus \mathcal{S}_{3} \\
                       \mathcal{S}_{3} \oplus \mathcal{S}_{1} \oplus \mathcal{S}_{2} \\
                       \mathcal{S}_2
                     \end{array}.
$$
Applying the same argument to a lift of the map $\mathcal{P}_i \twoheadrightarrow \mathcal{S}_i \hookrightarrow \Rad(\mathcal{P}_{i-1})/ \Rad^2(\mathcal{P}_{i-1})$, we obtain
$$
\mathcal{P}_i \simeq \begin{array}{c}
                       \mathcal{S}_i \\
                       \mathcal{S}_{i} \oplus \mathcal{S}_{i-1} \oplus \mathcal{S}_{i+1} \\
                       \mathcal{S}_{i+1} \oplus \mathcal{S}_{i-1} \oplus \mathcal{S}_{i} \\
                       \mathcal{S}_i
                     \end{array},
$$
for $i=2,\ldots, \ell-1$. We now consider $\mathcal{P}_\ell$. Since $\mathcal{P}_\ell$ is self-dual, $\eqref{Eq: composition factors}$ implies that we have
$$\mathcal{P}_\ell \simeq \begin{array}{c}
                            \mathcal{S}_\ell \\
                            \mathcal{S}_{\ell-1} \\
                            \mathcal{S}_{\ell-1} \\
                            \mathcal{S}_\ell
                          \end{array}
\quad \text{ or } \quad
\mathcal{P}_\ell \simeq \begin{array}{c}
                            \mathcal{S}_\ell \\
                            \mathcal{S}_{\ell-1} \oplus \mathcal{S}_{\ell-1} \\
                            \mathcal{S}_\ell
                          \end{array} .
 $$
Let $\psi:\mathcal{P}_\ell \rightarrow \mathcal{P}_{\ell-1}$ be a lift of the map $\mathcal{P}_\ell \twoheadrightarrow \mathcal{S}_\ell \hookrightarrow \Rad(\mathcal{P}_{\ell-1})/ \Rad^2(\mathcal{P}_{\ell-1})$.
It follows from the shape of the radical series of $\mathcal{P}_{\ell-1}$ that $\Rad^2(\Im \psi)\simeq \mathcal{S}_{\ell-1}$, which
implies that $\mathcal{S}_{\ell-1}$ appears in $\Rad^2(\mathcal{P}_\ell)/ \Rad^3(\mathcal{P}_\ell)$. Therefore, we have
$$\mathcal{P}_\ell \simeq \begin{array}{c}
                            \mathcal{S}_\ell \\
                            \mathcal{S}_{\ell-1} \\
                            \mathcal{S}_{\ell-1} \\
                            \mathcal{S}_\ell
                          \end{array},
 $$
which completes the proof.
\end{proof}

\begin{lemma} \label{Lem: commutative}
If $\ell = 2$, then there is an isomorphism of algebras
$$ e(0121) \fqH(4) e(0121) \simeq \bR[x,y]/(x^2, y^2 - axy),$$
for some $a \in\bR$.
\end{lemma}
\begin{proof}
We have $\delta = \alpha_0 + 2\alpha_1 + \alpha_2$, for $\ell=2$.
Theorem \ref{Thm: dimension forumula} gives
\begin{align*}
& \dim  e(012)\fqH(\delta - \alpha_1) e(012) = \dim \fqH(\delta - \alpha_1) = 2, \\
& \dim e(0121)\fqH(\delta)e(0121) = 4.
\end{align*}
Since $\Lambda_0 - \delta + \alpha_{1} = r_{2}(\Lambda_0 - \alpha_0 - \alpha_1)$,
the argument in the proof of Lemma \ref{Lem: alg type} shows
$$ e(012)\fqH(\delta - \alpha_{1}) e(012) = \fqH(\delta - \alpha_{1}) \simeq \bR[x]/(x^2). $$
Thus, it follows from Theorem \ref{Thm: Ei Fi} and $E_{1} \fqH(\delta - \alpha_{1}) = 0$ that
we have an isomorphism of $\fqH(\delta - \alpha_{1})$-bimodules as follows.
$$
( \bR \oplus \bR y) \otimes  e(012)\fqH(\delta - \alpha_{1}) e(012) \simeq e(0121)\fqH(\delta ) e(0121).
$$
We conclude that $e(0121)\fqH(\delta ) e(0121)\simeq \bR[x,y]/(x^2, y^2 - axy)$, for some $a \in\bR$.
\end{proof}

\begin{thm} \label{Thm: wild for delta}
If $\ell = 2$, then the algebra $\fqH(\delta)$ is a symmetric special biserial algebra of tame representation type.
When $\ell \ge 3$, $\fqH(\delta)$ is of wild representation type.
\end{thm}
\begin{proof} Suppose that $\ell=2$. Proposition \ref{Prop: radical series for delta} gives
 \begin{align*}
\mathcal{P}_1 \simeq \begin{array}{c}
                       \mathcal{S}_1 \\
                       \mathcal{S}_1 \oplus \mathcal{S}_2 \\
                       \mathcal{S}_2 \oplus \mathcal{S}_1 \\
                       \mathcal{S}_1
                     \end{array}
,\qquad
\mathcal{P}_2 \simeq \begin{array}{c}
                       \mathcal{S}_2 \\
                       \mathcal{S}_{1}  \\
                       \mathcal{S}_{1}  \\
                       \mathcal{S}_2
                     \end{array}
,
\end{align*}
which imply
$$
\dim\Hom(\mathcal{P}_i, \Rad(\mathcal{P}_j)/ \Rad^2(\mathcal{P}_j)) =  \left\{
                                                          \begin{array}{ll}
                                                            1 & \hbox{ if } (i=1) \text{ or } (i=2, j=1), \\
                                                            0 & \hbox{ if } i=j=2.
                                                          \end{array}
                                                        \right.
$$

By $\eqref{Eq: P and Q}$, $\mathcal{P}_1$ has a submodule $\mathcal{Q}$ which is isomorphic to $\mathcal{Q}_1$.
Let $\gamma: \mathcal{P}_1 \twoheadrightarrow \mathcal{Q} \hookrightarrow \mathcal{P}_1$ be the homomorphism induced from $\eqref{Eq: P and Q}$.
Note that $\gamma$ is a lift of $\mathcal{P}_1 \twoheadrightarrow \mathcal{S}_1 \hookrightarrow \Rad(\mathcal{P}_1)/ \Rad^2(\mathcal{P}_1)$. Since
$ \Im (\gamma) = \Ker(\gamma) \simeq \mathcal{Q}_1$, we have $\gamma^2 = 0$. We set
\begin{align*}
\alpha = \text{ a lift of $\mathcal{P}_1 \twoheadrightarrow \mathcal{S}_1 \hookrightarrow \Rad(\mathcal{P}_2)/ \Rad^2(\mathcal{P}_2) $}, \\
\beta = \text{ a lift of $\mathcal{P}_2 \twoheadrightarrow \mathcal{S}_2 \hookrightarrow \Rad(\mathcal{P}_1)/ \Rad^2(\mathcal{P}_1) $}.
\end{align*}
$\Im(\beta)$ is uniserial since $\mathcal{P}_2$ is. Considering the configuration of the radical series, we have
\begin{align*}
\Im(\alpha) = \Rad(\mathcal{P}_2)
, \quad
\Ker(\alpha) \simeq \begin{array}{c}
                       \mathcal{S}_2 \\
                       \mathcal{S}_1 \\
                       \mathcal{S}_1
                     \end{array}
, \quad
\Im(\beta) \simeq \begin{array}{c}
                       \mathcal{S}_2 \\
                       \mathcal{S}_1 \\
                       \mathcal{S}_1
                     \end{array}
, \quad
\Ker(\beta) \simeq \mathcal{S}_2.
\end{align*}
Thus, $\beta \alpha = 0$ and $ \Im( \gamma \alpha \beta ) = \Soc(\mathcal{P}_1) = \Im( \alpha \beta \gamma ) $.

By Theorem \ref{Thm: dimension forumula}, we have $\dim \fqH(\alpha_0 + \alpha_1 + \alpha_2) e(012) = 2$.
On the other hand, $\dim\mathcal{M}_1 =2$ by $\dim \mathcal{L}_1=|\ST(\lambda^{(1)})|=1$ and
we have a surjective homomorphism
$$ \fqH(\alpha_0 + \alpha_1 + \alpha_2) e(012)\to \mathcal{M}_1 $$
by $e(012)\mathcal{L}_1\ne0$. Since $\mathcal{M}_1$ is projective, it is a split epimorphism.
We have $ \mathcal{M}_1 \simeq F_2 F_1 F_0 \mathbf{1} $, where $\mathbf{1}$ is the trivial $\fqH(0)$-module.
Thus, we have $ \mathcal{P}_1 \simeq F_1 F_2 F_1 F_0 \mathbf{1}$. Lemma \ref{Lem: commutative} shows that $\End(\mathcal{P}_1) \simeq e(0121) \fqH(\delta)e(0121)$ is commutative, which yields
$$ \gamma \alpha \beta = \alpha \beta \gamma . $$
Therefore, the quiver of the basic algebra of $\fqH(\delta)$ is given as
$$
\bigskip
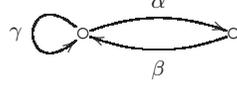
\hspace{4mm}
\begin{xy}
(10,0) *{\circ}="A", (30,0) *{\circ}="B",
\ar @(lu,ld) "A";"A"_{\gamma}
\ar @/^/ "A";"B"^{\alpha}
\ar @/^/ "B";"A"^{\beta}
\end{xy}
$$
and the defining relations are
$$
\beta \alpha = 0, \qquad \gamma \alpha \beta = \alpha \beta \gamma, \qquad \gamma^2 = 0.
$$
The assertion follows by \cite[Thm.\ 6.1 (2b)]{AIP13}.

Suppose that $\ell \ge 3$. Considering the configuration of the radical series in Proposition \ref{Prop: radical series for delta}, the quiver
of the basic algebra of $\fqH(\delta)$ has $\ell$ vertices and it is given as follows.
$$
\bigskip
\hspace{4mm}
\begin{xy}
(10,0) *{\circ}="A", (30,0) *{\circ}="B", (50,0) *{\circ}="C", (90,0) *{\circ}="D", (110,0) *{\circ}="E",
(70,0) *{\cdots\cdots\cdots\cdots},
\ar @(lu,ld) "A";"A"
\ar @(lu,ru) "B";"B"
\ar @(lu,ru) "C";"C"
\ar @(lu,ru) "D";"D"
\ar @/^/ "A";"B"
\ar @/^/ "B";"A"
\ar @/^/ "B";"C"
\ar @/^/ "C";"B"
\ar @/^/ "D";"E"
\ar @/^/ "E";"D"
\end{xy}
\ $$
Then, the assertion follows by \cite[I.10.8(iv)]{Erd90}.
\end{proof}

\section{Representations of $\fqH(2\delta - \varpi_4)$}

In this section, we assume that $\ell \ge 4$.
Let
$$\beta_0 := 2\delta - \varpi_4 = 2 \alpha_0 + 3 \alpha_1 + 2\alpha_2 + \alpha_3. $$
Using the crystal of the Fock space in Section \ref{Sec: Fock space}, $B(\Lambda_0)_{\Lambda_0 - \beta_0} $ has two elements $b_1$, $b_2$, which are realized as
the following Young diagrams:
$$
b_1 =
\xy
(0,12)*{};(18,12)*{} **\dir{-};
(0,6)*{};(18,6)*{} **\dir{-};
(0,0)*{};(12,0)*{} **\dir{-};
(0,-6)*{};(12,-6)*{} **\dir{-};
(0,-12)*{};(6,-12)*{} **\dir{-};
(0,12)*{};(0,-12)*{} **\dir{-};
(6,12)*{};(6,-12)*{} **\dir{-};
(12,12)*{};(12,-6)*{} **\dir{-};
(18,12)*{};(18,6)*{} **\dir{-};
\endxy \ ,
\qquad
b_2 =
\xy
(0,12)*{};(12,12)*{} **\dir{-};
(0,6)*{};(12,6)*{} **\dir{-};
(0,0)*{};(12,0)*{} **\dir{-};
(0,-6)*{};(12,-6)*{} **\dir{-};
(0,-12)*{};(12,-12)*{} **\dir{-};
(0,12)*{};(0,-12)*{} **\dir{-};
(6,12)*{};(6,-12)*{} **\dir{-};
(12,12)*{};(12,-12)*{} **\dir{-};
\endxy\ .
$$
Note that
\begin{align} \label{Eq: epsilon of T}
\varepsilon_i(b_1) = \left\{
                                   \begin{array}{ll}
                                     1 & \hbox{ if } i=1,3, \\
                                     0 & \hbox{ otherwise},
                                   \end{array}
                                 \right.
\qquad
\varepsilon_i(b_2) = \left\{
                                   \begin{array}{ll}
                                     1 & \hbox{ if } i=2, \\
                                     0 & \hbox{ otherwise}.
                                   \end{array}
                                 \right.
\end{align}
We denote by $\mathcal{T}_1$ and $\mathcal{T}_2$ the irreducible $\fqH(\beta_0)$-modules which corresponds to $b_1$ and $b_2$ respectively.

On the other hand, $\Lambda_0 - \beta_0 + \alpha_0$ is not a weight of $V(\Lambda_0)$ by Theorem \ref{Thm: dimension forumula}.
Then, by direct computations, we have
\begin{align*}
\Lambda_0 - \beta_0 + \alpha_{3} = r_2r_1r_0r_1r_2(\Lambda_0 - \alpha_0 - \alpha_1),\\
\Lambda_0 - \beta_0 + \alpha_2 = r_3r_1r_0r_1r_2(\Lambda_0 - \alpha_0 - \alpha_1),\\
\Lambda_0 - \beta_0 + \alpha_{1} = r_2r_3r_0r_1r_2(\Lambda_0 - \alpha_0 - \alpha_1),
\end{align*}
and the algebras $\fqH(\beta_0 - \alpha_k)$, for $k=1,2,3$, are derived equivalent to $\fqH(\alpha_0 + \alpha_1)$.
Since $\fqH(\alpha_0 + \alpha_1) \simeq \bR[x]/(x^2)$, $\fqH(\beta_0 - \alpha_k)$ are matrix rings over $\bR[x]/(x^2)$
by the same argument as in Lemma \ref{Lem: alg type}. Similarly, it follows from
\begin{align*}
\Lambda_0 - \beta_0 + \alpha_1 + \alpha_2 = r_{3}r_0r_1r_2(\Lambda_0 - \alpha_0 - \alpha_1),\\
\Lambda_0 - \beta_0 + \alpha_2 + \alpha_3 = r_1r_0r_1r_2(\Lambda_0 - \alpha_0 - \alpha_1)
\end{align*}
that $\fqH( \beta_0 - \alpha_1 - \alpha_2)$ and $\fqH( \beta_0 - \alpha_2 - \alpha_3)$ are isomorphic to matrix rings over $\bR[x]/(x^2)$.

For $k=1,2,3,$ let $\mathcal{U}_k$ be the unique irreducible $\fqH(\Lambda_0 - \beta_0 + \alpha_k)$-module and $\widehat{\mathcal{U}}_k$ its projective cover.
Note that $\widehat{\mathcal{U}}_k$ has the radical series
\begin{align} \label{Eq: projective1}
\widehat{\mathcal{U}}_k \simeq \begin{array}{c}
                          \mathcal{U}_k \\
                          \mathcal{U}_k
                        \end{array}
.
\end{align}
By $\eqref{Eq: epsilon of T}$, we may apply Lemma \ref{irreducibility} to $\mathcal{T}_1$ and $\mathcal{T}_2$. Then the uniqueness of the irreducible $\fqH(\Lambda_0 - \beta_0 + \alpha_k)$-modules implies that
$$ E_2(\mathcal{T}_2) \simeq \mathcal{U}_2, \quad E_1(\mathcal{T}_1) \simeq \mathcal{U}_1, \quad E_3(\mathcal{T}_1) \simeq \mathcal{U}_3. $$

We consider the following projective $\fqH(\beta_0)$-modules
$$ \mathcal{R}_i := F_i \widehat{\mathcal{U}}_i, \quad\text{for $i=1,2,3$}.$$
Then, by the biadjointness of $E_i$ and $F_i$,
\begin{align*}
\dim\Hom(\mathcal{R}_i, \mathcal{T}_1) = \dim\Hom(\widehat{\mathcal{U}}_i, E_i \mathcal{T}_1) = \left\{
                                                                                                              \begin{array}{ll}
                                                                                                                1 & \hbox{ if } i=1,3, \\
                                                                                                                0 & \hbox{ otherwise,}
                                                                                                              \end{array}
                                                                                                            \right.
 \\
\dim\Hom(\mathcal{R}_i, \mathcal{T}_2) = \dim\Hom(\widehat{\mathcal{U}}_i, E_i \mathcal{T}_2) = \left\{
                                                                                                      \begin{array}{ll}
                                                                                                        1 & \hbox{ if } i =2, \\
                                                                                                        0 & \hbox{ otherwise, }
                                                                                                      \end{array}
                                                                                                    \right.
\\
\dim\Hom(\mathcal{T}_1, \mathcal{R}_i) = \dim\Hom(E_i \mathcal{T}_1, \widehat{\mathcal{U}}_i) = \left\{
                                                                                                      \begin{array}{ll}
                                                                                                        1 & \hbox{ if } i =1,3, \\
                                                                                                        0 & \hbox{ otherwise, }
                                                                                                      \end{array}
                                                                                                    \right.
\\
\dim\Hom(\mathcal{T}_2, \mathcal{R}_i) = \dim\Hom(E_i \mathcal{T}_2, \widehat{\mathcal{U}}_i) = \left\{
                                                                                                      \begin{array}{ll}
                                                                                                        1 & \hbox{ if } i =2, \\
                                                                                                        0 & \hbox{ otherwise. }
                                                                                                      \end{array}
                                                                                                    \right.
\end{align*}
Thus, $\mathcal{R}_2$ is the projective cover of $\mathcal{T}_2$. Since both of $\mathcal{R}_1$ and $\mathcal{R}_3$ are indecomposable
projective modules which surjects to $\mathcal{T}_1$, $\mathcal{R}_1 \simeq \mathcal{R}_3$ is the projective cover of $\mathcal{T}_1$.

In the crystal of the Fock space $\F$, we have $\varepsilon_1(\mathcal{U}_2) = \varepsilon_2(\mathcal{U}_1) = 1$.
Thus, Lemma \ref{irreducibility} implies that $E_1(\mathcal{U}_2)$ and $E_2(\mathcal{U}_1)$ are irreducible
$\fqH(\beta_0-\alpha_1-\alpha_2)$-modules, and the uniqueness of the irreducible $\fqH(\beta_0-\alpha_1-\alpha_2)$-modules
implies $ E_1(\mathcal{U}_2) \simeq E_2(\mathcal{U}_1) $. We have the exact sequence
\begin{align}
0 \rightarrow E_1\mathcal{U}_2 \rightarrow E_1\widehat{\mathcal{U}}_2 \rightarrow E_1\mathcal{U}_2 \rightarrow 0 .
\end{align}
Since $E_1(\mathcal{U}_2)$ is not projective, it does not split, and $E_1\widehat{\mathcal{U}}_2$ is indecomposable projective.
The same argument shows that $E_2\widehat{\mathcal{U}}_1$ is indecomposable projective.
Hence, the indecomposable projective $\fqH(\beta_0 - \alpha_1 - \alpha_2)$-module is given by
$$
E_1(\widehat{\mathcal{U}}_2) \simeq E_2(\widehat{\mathcal{U}}_1) \simeq \begin{array}{c}
                                                                           E_1(\mathcal{U}_2) \\
                                                                           E_1(\mathcal{U}_2)
                                                                         \end{array}.
$$
It follows that, for $i,j = 1,2$ with $i \ne j$, we have
\begin{align*}
&\dim \Hom( \mathcal{R}_i, \mathcal{R}_j) = \dim \Hom( E_j \widehat{\mathcal{U}}_i, E_i \widehat{\mathcal{U}}_j) = 2, \\
&\dim \Hom( \mathcal{R}_i, \mathcal{R}_i) = \dim \Hom( \widehat{\mathcal{U}}_i, E_i F_i \widehat{\mathcal{U}}_i)
= \dim \Hom( \widehat{\mathcal{U}}_i, \widehat{\mathcal{U}}_i ^{ \oplus \langle h_i, \Lambda_0 - \beta_0 + \alpha_i \rangle  }) = 4.
\end{align*}
Therefore, $\mathcal{R}_1$ and $\mathcal{R}_2$ are self-dual modules whose composition multiplicities are given by
\begin{align*}
[\mathcal{R}_1] = 4 [\mathcal{T}_1] + 2 [\mathcal{T}_2], \qquad  [\mathcal{R}_1] = 2 [\mathcal{T}_1] + 4 [\mathcal{T}_2].
\end{align*}

Let $\mathcal{V}_i := F_i \mathcal{U}_i$, for $i=1,2$. By the same argument as above, we have
\begin{align*}
\dim \Hom( \mathcal{V}_i, \mathcal{T}_j) = \dim \Hom( \mathcal{T}_i, \mathcal{V}_j) = \delta_{i,j}.
\end{align*}
We have the exact sequence
\begin{align} \label{Eq: nonsplit P R}
0 \rightarrow \mathcal{V}_i \rightarrow \mathcal{R}_i \rightarrow \mathcal{V}_i \rightarrow 0,
\end{align}
which does not split because $\mathcal{R}_i$ are indecomposable. As $\Top(\mathcal{V}_i)\simeq\mathcal{T}_i\simeq \Soc(\mathcal{V}_i)$, we have
\begin{align} \label{Eq: radicals for V}
\mathcal{V}_1 \simeq \begin{array}{c}
                       \mathcal{T}_1 \\
                       \mathcal{T}_2  \\
                       \mathcal{T}_1
                     \end{array} ,
\qquad \qquad
\mathcal{V}_2 \simeq \begin{array}{c}
                       \mathcal{T}_2 \\
                       \mathcal{T}_1  \\
                       \mathcal{T}_2
                     \end{array}.
\end{align}

\begin{prop} \label{Prop: radical series for 2delta-pi4}
The radical series of $\mathcal{R}_1$ and $\mathcal{R}_2$ are given as follows:
\begin{align*}
\mathcal{R}_1 \simeq \begin{array}{c}
                       \mathcal{T}_1 \\
                       \mathcal{T}_1 \oplus \mathcal{T}_2 \\
                       \mathcal{T}_2 \oplus \mathcal{T}_1 \\
                       \mathcal{T}_1
                     \end{array}
,\qquad \qquad
\mathcal{R}_2 \simeq \begin{array}{c}
                       \mathcal{T}_2 \\
                       \mathcal{T}_2 \oplus \mathcal{T}_1 \\
                       \mathcal{T}_1 \oplus \mathcal{T}_2 \\
                       \mathcal{T}_2
                     \end{array}
\end{align*}
\end{prop}
\begin{proof}
As the argument is symmetric in $i=1$ and $i=2$, we only consider $\mathcal{R}_1$. It is clear from (\ref{Eq: radicals for V}) that
$\mathcal{T}_2$ appears in $\Rad(\mathcal{R}_1)/\Rad^2(\mathcal{R}_1)$. If $\Rad(\mathcal{R}_1)/\Rad^2(\mathcal{R}_1)$ is irreducible,
then $\Ext^1(\mathcal{T}_1, \mathcal{T}_1)=0$. Since $\Rad^2(\mathcal{R}_1)/\Rad^3(\mathcal{R}_1)$ contains $\mathcal{T}_1$ by
(\ref{Eq: nonsplit P R}) and (\ref{Eq: radicals for V}), it implies that $\mathcal{R}_1$ has the radical series of the following form.
$$
\mathcal{R}_1 \simeq \begin{array}{c}
                       \mathcal{T}_1 \\
                       \mathcal{T}_2 \\
                       \mathcal{T}_1 \oplus \mathcal{T}_1 \\
                       \mathcal{T}_2 \\
                       \mathcal{T}_1
                     \end{array}
$$
But if we look at $\Rad(\mathcal{R}_1)/\Rad^3(\mathcal{R}_1)$, we have $\dim \Ext{^1}(\mathcal{T}_2,\mathcal{T}_1)\ge 2$, and the self-duality of irreducible
modules implies that $\dim \Ext{^1}(\mathcal{T}_1,\mathcal{T}_2)\ge 2$. It contradicts
$\dim \Ext{^1}(\mathcal{T}_1,\mathcal{T}_2)=1$.
Thus, $\Rad(\mathcal{R}_1)/\Rad^2(\mathcal{R}_1)$ is not irreducible, and we have the desired shape of the radical series.
\end{proof}


\begin{thm} \label{Thm: wild for 2delta-pi4}
The algebra $\fqH(2\delta - \varpi_4)$ is wild.
\end{thm}
\begin{proof}
By $\eqref{Eq: nonsplit P R}$, $\mathcal{R}_1$ has a submodule $\mathcal{V}$ which is isomorphic to $\mathcal{V}_1$.
Let $\gamma: \mathcal{R}_1 \twoheadrightarrow \mathcal{V} \hookrightarrow \mathcal{R}_1$ be the homomorphism induced by $\eqref{Eq: nonsplit P R}$,
which is a lift of $\mathcal{R}_1 \twoheadrightarrow \mathcal{T}_1 \hookrightarrow \Rad(R_1)/ \Rad^2(R_1)$.
We have $\gamma^2 = 0$. Similarly,
we take a lift $\delta$ of $\mathcal{R}_2 \twoheadrightarrow \mathcal{T}_2 \hookrightarrow \Rad(R_2)/ \Rad^2(R_2)$ such that $\delta^2 = 0$.
We now choose
\begin{align*}
\alpha = \text{ a lift of $\mathcal{R}_1 \twoheadrightarrow \mathcal{T}_2 \hookrightarrow \Rad(\mathcal{R}_2)/ \Rad^2(\mathcal{R}_2) $}, \\
\beta = \text{ a lift of $\mathcal{R}_2 \twoheadrightarrow \mathcal{T}_1 \hookrightarrow \Rad(\mathcal{R}_1)/ \Rad^2(\mathcal{R}_1) $}.
\end{align*}
Then, the quiver of the basic algebra of $\fqH(2\delta - \varpi_4)$ is given as follows:
\begin{align} \label{Eq: quiver}
\bigskip
\hspace{4mm}
\begin{xy}
(10,0) *{\circ}="A", (30,0) *{\circ}="B",
\ar @(lu,ld) "A";"A"_{\gamma}
\ar @/^/ "A";"B"^{\alpha}
\ar @/^/ "B";"A"^{\beta}
\ar @(ru,rd) "B";"B"^{\delta}
\end{xy}
\end{align}
Considering the configuration of the radical series from Proposition \ref{Prop: radical series for 2delta-pi4}, we must have
$$
\Im(\alpha\beta) \simeq \begin{array}{c}
                                  \mathcal{T}_1 \\
                                  \mathcal{T}_1
                          \end{array}, \quad
\Im(\beta\alpha) \simeq \begin{array}{c}
                                  \mathcal{T}_2 \\
                                  \mathcal{T}_2
                          \end{array}, \quad
\Im(\alpha) \simeq \begin{array}{c}
                                  \mathcal{T}_1 \\
                                  \mathcal{T}_1\oplus \mathcal{T}_2 \\
                                  \mathcal{T}_2
                          \end{array}, \quad
\Im(\beta) \simeq \begin{array}{c}
                                  \mathcal{T}_2 \\
                                  \mathcal{T}_2\oplus \mathcal{T}_1 \\
                                  \mathcal{T}_1
                          \end{array}\quad
$$
and it follows that
\begin{align*}
\alpha\beta\alpha = \beta\alpha\beta = 0,\qquad \Im(\gamma \alpha) =  \Im( \alpha \delta ) \simeq \begin{array}{c}
                                                                                                        \mathcal{T}_1 \\
                                                                                                        \mathcal{T}_2
                                                                                                      \end{array}
,\qquad  \Im(\delta \beta) =  \Im( \beta \gamma ) \simeq \begin{array}{c}
                                                                    \mathcal{T}_2 \\
                                                                    \mathcal{T}_1
                                                                    \end{array}
.
\end{align*}
By adjusting $\gamma$ and $\delta$ by nonzero scalar multiples, we may assume $ \gamma \alpha  =  \alpha \delta $. Thus, we have the defining relations for the basic algebra
as follows, where $c\in \bR$ is a nonzero scalar:
$$
\gamma^2 = \delta^2 = \alpha\beta\alpha = \beta\alpha\beta = 0, \quad \gamma \alpha  =  \alpha \delta, \quad \delta \beta  = c  \beta \gamma.
$$
Since the algebra of the quiver $\eqref{Eq: quiver}$ with the defining relations
$$
\gamma^2 = \delta^2 = \alpha\beta\alpha = \beta\alpha\beta = 0, \quad \gamma \alpha  =  \alpha \delta, \quad \delta \beta  =  \beta \gamma = 0
$$
is of wild representation type by \cite[Thm.\ 1,\ Table W (32)]{Han02}, so is $\fqH(2\delta - \varpi_4)$.
\end{proof}

\section{Representations type of $\fqH(\beta)$}

By the categorification theorem, $R^{\Lambda}(\beta) \ne 0$ if and only if $\Lambda - \beta$ is a weight of $V(\Lambda)$.
A weight $\mu$ of $V(\Lambda)$ is {\it maximal} if $\mu + \delta$ is not a weight of $V(\Lambda)$. Let $\max(\Lambda)$ be the set of all maximal weights of $V(\Lambda)$.

\begin{prop} \label{Prop: maximal}
For the weight system of the $\g(A)$-module $V(\Lambda_0)$ in type $C_\ell^{(1)}$, we have
\begin{enumerate}
\item $ \max(\Lambda_0) \cap \wlP^+ = \{  \Lambda_0 + \varpi_i - \frac{i}{2}\delta \mid i\in I,\ \text{$i$ is even } \},$
\item $\mu$ is a weight of $V(\Lambda_0)$ if and only if $\mu = w \eta - k \delta$ for some $w\in \weyl$, $\eta \in \max(\Lambda_0) \cap \wlP^+ $ and $k\in \Z_{\ge 0}$.
\end{enumerate}
\end{prop}
\begin{proof}
(1) Let $\mu \in \max(\Lambda_0) \cap \wlP^+$. Since $\mu \in \wlP^+$ and $\varpi_1$, \ldots, $\varpi_\ell$ form a basis of $\sum_{i\in I\setminus \{0\}} \Q \alpha_i$,
$\mu$ can be written as
$$\mu = \Lambda_0 + \sum_{i\in I\setminus \{ 0\}} p_i \varpi_i + t \delta $$
for some $ p_i =\mu(h_i) \in \Z_{\ge0}$ and $t \in \Z$. Then, the computations
\begin{align*}
0 &\le \mu(h_0) = 1-p_1 - \cdots - p_n, \\
0 &\le \mu(h_1 + \cdots + h_n) =  p_1 + \cdots + p_n
\end{align*}
imply that $\mu = \Lambda_0 + \varpi_i + t \delta $ for some $i \in I\setminus \{ 0 \}$,
or $\mu=\Lambda_0 + t \delta$. In the latter case, $\mu \in \max(\Lambda_0)$ implies that $\mu=\Lambda_0$, which is equal to $\Lambda_0+\varpi_0$. In the former case,
$\Lambda_0 - \mu \in \rlQ^+$ implies that $i$ is even by the definition $\eqref{Eq: pi}$.
We show that $t = - \frac{i}{2}$. We consider the Young diagram
$$ \lambda(i) = (\underbrace{i,i,\ldots, i}_{i/2}) $$
in the Fock space $\F$. Considering the residue pattern, we have
$$\wt(\lambda(i)) = \Lambda_0 - \left( \frac{i}{2} \alpha_0 + (i-1) \alpha_1 + (i-2)\alpha_2 + \cdots + \alpha_{i-1}  \right) = \Lambda_0 + \varpi_i - \frac{i}{2}\delta.$$
Thus, Theorem \ref{Thm: dimension forumula} implies
$$ \dim \fqH( \frac{i}{2}\delta - \varpi_i ) \ne 0,$$
and $\Lambda_0 + \varpi_i - \frac{i}{2}\delta$ is a weight of $V(\Lambda_0)$.
It follows from
$$ (- \varpi_i + \frac{i}{2}\delta) - \delta \notin \rlQ^+ $$
that
$\Lambda_0 + \varpi_i - \frac{i}{2}\delta$ is maximal.

(2) $ \max(\Lambda_0)$ is $\weyl$-invariant by \cite[Prop.\ 10.1]{Kac90} and we have
$$ \max(\Lambda_0) = \weyl (\max(\Lambda_0) \cap \wlP^+) $$
by \cite[Cor.\ 10.1]{Kac90}. Then, for any weight $\mu$ of $V(\Lambda_0)$,
there exist a unique $ \zeta \in \max(\Lambda_0)$ and a unique $k \in \Z_{\ge0}$ such that
$ \mu = \zeta - k\delta$ \cite[(12.6.1)]{Kac90}.
\end{proof}

\begin{lemma}[{\cite[Prop.2.3]{EN02}}, {\cite[Remark.5.10]{AP13}}]
\label{Lem: reduction to critical rank}
Let $A$ and $B$ be finite dimensional $\bR$-algebras and suppose that
there exists a constant $C>0$ and functors
$$
F:\;A\text{\rm -mod} \rightarrow B\text{\rm -mod}, \quad
G:\;B\text{\rm -mod} \rightarrow A\text{\rm -mod}
$$
such that, for any $A$-module $M$,
\begin{itemize}
\item[(1)]
$M$ is a direct summand of $GF(M)$ as an $A$-module,
\item[(2)]
$\dim F(M)\le C\dim M$.
\end{itemize}
Then, if $A$ is wild, so is $B$.
\end{lemma}

\begin{lemma} \label{Lem: wild}
\begin{enumerate}
\item [(1)] If $\fqH(\beta - \alpha_j)$ is wild and $\langle h_j, \Lambda_0 - \beta + \alpha_j \rangle \ge 1$, then $\fqH(\beta)$ is wild.
\item[(2)] Suppose that $\fqH(k \delta - \varpi_i)$ is wild. Then, we have
\begin{enumerate}
\item[(a)] $\fqH((k + 1)\delta - \varpi_{i})$ is wild,
\item[(b)] if $i+2 \in I$, then $\fqH( (k+1)\delta - \varpi_{i+2})$ is wild.
\end{enumerate}
\end{enumerate}
\end{lemma}
\begin{proof}
(1) Considering the functors
$$F_j: \fqH(\beta - \alpha_j)\text{-mod} \rightarrow \fqH(\beta)\text{-mod}, \quad E_j: \fqH(\beta)\text{-mod} \rightarrow \fqH(\beta- \alpha_j)\text{-mod}, $$
the assertion follows from Lemma \ref{Lem: reduction to critical rank} and Theorem \ref{Thm: Ei Fi}.

(2) For $0\le i \le \ell-1$ and $k\in \Z_{\ge0}$, direct computation shows
\begin{align*}
\Lambda_0 + \varpi_{i+2} - (k+1)\delta + \alpha_{i+1} &= r_{i}r_{i-1}\cdots r_{1} r_{0} r_{1} \cdots r_{i} (\Lambda_0 + \varpi_i - k\delta),\\
\Lambda_0 + \varpi_i - (k+1)\delta + \alpha_\ell &= r_{\ell-1}r_{\ell-2}\cdots r_{1} r_{0} r_{1} \cdots r_{i} (\Lambda_0 + \varpi_i - k\delta).
\end{align*}
Thus, (2)(a), for $i\ne\ell$, and (2)(b) follow from Proposition \ref{Prop: repn type} and (1) because
\begin{align*}
&\langle h_{i+1}, \Lambda_0 + \varpi_{i+2} - (k+1)\delta + \alpha_{i+1} \rangle = 2,\\
&\langle h_{\ell}, \Lambda_0 + \varpi_i - (k+1)\delta + \alpha_\ell \rangle = 2.
\end{align*}
Similarly, we consider
$$ \Lambda_0 + \varpi_\ell - (k+1)\delta + \alpha_0 = r_{1}r_{2} \cdots r_{\ell} (\Lambda_0 + \varpi_\ell - k\delta). $$
Then (2)(a), for $i=\ell$, follows from Proposition \ref{Prop: repn type} and (1) because
$$ \langle h_0, \Lambda_0 + \varpi_\ell - (k+1)\delta + \alpha_0 \rangle = 2. $$
We have proved the lemma.
\end{proof}

\begin{lemma} \label{Lem: wild2}
The algebras $\fqH(2\delta - \varpi_2)$ and $\fqH(2\delta)$ are wild.
\end{lemma}
\begin{proof}
Note that $2\delta - \varpi_2  = \delta + \alpha_0 + \alpha_1 $.
If $\ell \ge 3$, Lemma \ref{Lem: wild}(1) and Theorem \ref{Thm: wild for delta} imply that $\fqH(2\delta - \varpi_2)$ is wild, because we have
$$
\langle h_0, \Lambda_0 - \delta  \rangle = 1,\quad
\langle h_1, \Lambda_0 - \delta - \alpha_0  \rangle = 2.
$$
Applying Lemma \ref{Lem: wild}(2)(a), Theorem \ref{Thm: wild for delta} also implies that $\fqH(2\delta)$ is wild.

In the following, we suppose that $\ell = 2$. We set
$$ e_0 = \sum_{\nu \in I^{\delta}} e(\nu, 0),\quad e_1 = \sum_{\nu' \in I^{\delta + \alpha_0}} e(\nu', 1), \quad e = \sum_{\nu \in I^{\delta}} e(\nu, 0, 1). $$
Considering the residue pattern and Theorem \ref{Thm: dimension forumula}, we have
$$ E_0 \fqH(\delta) = 0. $$
Since $\langle h_0, \Lambda_0 - \delta \rangle = 1$, Theorem \ref{Thm: Ei Fi} gives an algebra isomorphism
$$ \fqH(\delta) \simeq E_0F_0 \fqH(\delta) = e_0\fqH(\delta + \alpha_0)e_0 . $$
We also have $  E_1 \fqH(\delta + \alpha_0  ) = 0 $ by Theorem \ref{Thm: dimension forumula}.
It follows from
$$ \langle h_1, \Lambda_0 - \delta - \alpha_0 \rangle = 2 $$
and Theorem \ref{Thm: Ei Fi} that there is a bimodule isomorphism
\begin{align} \label{Eq: isom1}
\bR[t]/(t^2) \otimes_\bR  \fqH(\delta + \alpha_0) \simeq E_1F_1 \fqH(\delta + \alpha_0 ) = e_1\fqH(2\delta - \varpi_2)e_1 .
\end{align}
Thus, multiplying $e=ee_1=e_1e$ on the both sides and factoring out the square of the radicals, $\eqref{Eq: isom1}$ gives the isomorphism
of algebras
\begin{align*}
e \fqH(2\delta - \varpi_2 ) e &/ \Rad^2( e \fqH(2\delta - \varpi_2 ) e ) \\
& \simeq \bR[t]/(t^2) \otimes_\bR \fqH(\delta)/ ( t^2, t\Rad(\fqH(\delta)), \Rad^2(\fqH(\delta))  ).
\end{align*}
We denote the algebra by $B$. Let $\mathcal{O}$ be the irreducible module $\bR[t]/(t^2)$-module. Then $B$
has irreducible modules $\mathcal{O}\otimes \mathcal{S}_1 $ and $\mathcal{O}\otimes \mathcal{S}_2$, where $\mathcal{S}_1$ and $\mathcal{S}_2$
are the irreducible $\fqH(\delta)$-modules in Proposition \ref{Prop: simples for delta}. By Proposition \ref{Prop: radical series for delta},
the projective cover of $\mathcal{O}\otimes \mathcal{S}_1 $ has the radical series
\begin{align*}
&\mathcal{O}\otimes \mathcal{S}_1  \\
\mathcal{O}\otimes \mathcal{S}_1 \quad & \mathcal{O}\otimes \mathcal{S}_1 \quad \mathcal{O}\otimes \mathcal{S}_2,
\end{align*}
which implies that the quiver of $e \fqH(2\delta - \varpi_2 ) e$ contains
\begin{align*}
\bigskip
\hspace{4mm}
\begin{xy}
(10,0) *{\circ}="A", (30,0) *{\circ}="B",
\ar @(u,l) "A";"A"
\ar @(l,d) "A";"A"
\ar  "A";"B"
\end{xy}
\end{align*}
as a subquiver. By \cite[I.10.8(i)]{Erd90}, $e \fqH(2\delta - \varpi_2 ) e$ is wild, and so is $\fqH(2\delta - \varpi_2 )$.
Then,  $\fqH(2\delta)=\fqH(2\delta - \varpi_2  + \alpha_1 + \alpha_2)$ is wild by Lemma \ref{Lem: wild} (1) because we have
$$ \langle h_2, \Lambda_0 - 2\delta + \alpha_1 + \alpha_2 \rangle =1, \quad
\langle h_1, \Lambda_0 - 2\delta + \alpha_1  \rangle =2. $$
We have proved the lemma.
\end{proof}

We summarize the results which are obtained so far. Suppose that $i\ge4$ is even. Then, Theorem \ref{Thm: wild for 2delta-pi4} and
Lemma \ref{Lem: wild}(2)(a)(b) imply that $\fqH( k\delta- \varpi_i)$, for $k\ge i/2$, are all wild.
If $i=2$, then $\fqH(\delta - \varpi_2)=\fqH(\alpha_0+\alpha_1)$ is of finite type by Lemma \ref{Lem: alg type} (1), and $\fqH(k\delta - \varpi_2)$, for $k\ge2$,
are wild by Lemma \ref{Lem: wild2} and Lemma \ref{Lem: wild}(2)(a). If $i=0$, $\fqH(0)$ is a simple algebra, and
$\fqH(\delta)$ is tame if $\ell = 2$ and wild if $\ell > 2$ by Theorem \ref{Thm: wild for delta}. As $\fqH(2\delta)$ is wild by
Lemma \ref{Lem: wild2}, Lemma \ref{Lem: wild}(2)(a) implies that $\fqH(k\delta)$, for $k\ge2$, are wild.
Thus, we have the following theorem.

\begin{thm} \label{Thm: main thm}
Let $i \in I$ be an even index.
For $\kappa \in \weyl(\Lambda_0 - \varpi_{i})$ and $ k \ge i/2$, the finite quiver Hecke algebra $\fqH( \Lambda_0 - \kappa +  k\delta )$ of type $C_\ell^{(1)}$ is
\begin{enumerate}
\item a simple algebra if $ i=k = 0$,
\item of finite representation type if $ i=2$ and $k=1$,
\item of tame representation type if $ i=0$, $k=1$ and $\ell=2$,
\item of wild representation type otherwise.
\end{enumerate}
\end{thm}


\bibliographystyle{amsplain}


\end{document}